\documentclass[11pt]{amsart}

\usepackage{array,float,hyperref,color}
\usepackage{amssymb, amsmath, amsthm, amsbsy, amscd, mathrsfs,stmaryrd}
\usepackage[normalem]{ulem}
\numberwithin{equation}{section}
\numberwithin{figure}{section}

\newcommand{\revone}[1]{{\color{black}{#1}}}
\newcommand{\revtwo}[1]{{\color{black}{#1}}}

\theoremstyle{plain}
 \newtheorem{theorem}{Theorem}[section]
 \newtheorem{proposition}[theorem]{Proposition}
 \newtheorem{lemma}[theorem]{Lemma}
 \newtheorem{corollary}[theorem]{Corollary}
 \newtheorem{condition}{Condition}[section]
 \newtheorem{conjecture}[theorem]{Conjecture}

\theoremstyle{definition}
 \newtheorem{definition}[theorem]{Definition}

\theoremstyle{remark}
 \newtheorem{assumption}{Assumption}[section]

 \newtheorem{remark}[theorem]{Remark}
 
 \newtheorem{question}[theorem]{Question}
 \newtheorem{example}[theorem]{Example}

\newcommand{\car}{\curvearrowright}
\newcommand{\nl}{\triangleleft}

\newcommand{\be}{\begin{enumerate}}
\newcommand{\ee}{\end{enumerate}}

\newcommand{\brem}{\begin{remark}}
\newcommand{\erem}{\end{remark}}

\newcommand{\bp}{\begin{proof}}
\newcommand{\ep}{\end{proof}}

\newcommand{\Fp}{\ensuremath{\mathbb{F}_p}}

\newcommand{\N}{\ensuremath{\mathbb{N}}}
\newcommand{\Zp}{\ensuremath{\mathbb{Z}_p}}

\newcommand{\C}{\ensuremath{\mathbb{C}}}
\newcommand{\R}{\ensuremath{\mathbb{R}}}
\newcommand{\Q}{\ensuremath{\mathbb{Q}}}
\newcommand{\Z}{\ensuremath{\mathbb{Z}}}

\newcommand{\bfo}{{\bf 1}}
\newcommand{\bfz}{{\bf 0}}

\newcommand{\bfbeta}{\boldsymbol{\beta}}

\newcommand{\bfe}{\ensuremath{\mathbf{e}}}
\newcommand{\bff}{\ensuremath{\mathbf{f}}}
\newcommand{\bfg}{\ensuremath{\mathbf{g}}}

\newcommand{\bfm}{\ensuremath{\mathbf{m}}}
\newcommand{\bfn}{\ensuremath{\mathbf{n}}}
\newcommand{\bfr}{\ensuremath{\mathbf{r}}}
\newcommand{\bfs}{\ensuremath{\mathbf{s}}}

\newcommand{\bfx}{{\boldsymbol{x}}}
\newcommand{\bfy}{\ensuremath{\mathbf{y}}}

\newcommand{\bfY}{\ensuremath{\mathbf{Y}}}

\newcommand{\mcB}{\mathcal{B}}
\newcommand{\mcD}{\mathcal{D}}
\newcommand{\mcE}{\mathcal{E}}
\newcommand{\mcH}{\mathcal{H}}
\newcommand{\mcO}{\mathcal{O}}
\newcommand{\mcR}{\mathcal{R}}

\newcommand{\mcM}{\mathcal{M}}
\newcommand{\mcN}{\ensuremath{\mathcal{N}}}
\newcommand{\mcV}{\ensuremath{\mathcal{V}}}

\newcommand{\mcL}{\ensuremath{\mathcal{L}}}

\newcommand{\mff}{\ensuremath{\mathfrak{f}}}
\newcommand{\mfg}{\ensuremath{\mathfrak{g}}}

\newcommand{\mfu}{\ensuremath{\mathfrak{u}}}

\newcommand{\tud}{\textup{d}}

\newcommand{\rk}{\ensuremath{{\rm rk}}}

\newcommand{\Ct}{\mathbb{C}\llbracket T \rrbracket}

\newcommand{\wt}{\ensuremath{\widetilde}}
\newcommand{\wtm}{\ensuremath{\widetilde{m}}}

\newcommand{\mcC}{{\mathcal{C}}}

\newcommand{\lri}{\mathfrak o}
\newcommand{\Lri}{\mathfrak O}

\newcommand{\la}{\langle}
\newcommand{\ra}{\rangle}

\newcommand{\Gri}{\ensuremath{\mathcal{O}}}

\newcommand{\wtdelta}{\widetilde{\delta}}
\newcommand{\mfp}{\mathfrak{p}}
\newcommand{\mfP}{\mathfrak{P}}

\newcommand{\rarr}{\rightarrow}

\newcommand{\udots}{\mathinner{\mskip1mu\raise1pt\vbox{\kern7pt\hbox{.}}
\mskip2mu\raise4pt\hbox{.}\mskip2mu\raise7pt\hbox{.}\mskip1mu}}

\DeclareMathOperator{\red}{red}
\DeclareMathOperator{\Fil}{Fil}

\DeclareMathOperator{\Cent}{Cent}
\DeclareMathOperator{\SubMod}{SubMod}
\DeclareMathOperator{\SL}{SL}
\DeclareMathOperator{\End}{End}

\DeclareMathOperator{\diag}{diag}

\DeclareMathOperator{\Id}{Id}

\DeclareMathOperator{\ad}{ad}

\DeclareMathOperator{\Mat}{Mat}

\DeclareMathOperator{\GL}{GL}

\DeclareMathOperator{\dep}{dep}

\renewcommand{\epsilon}{\varepsilon}
\renewcommand{\phi}{\varphi}
\begin{document}

\title[Local functional equations for submodule zeta functions]{Local
  functional equations for submodule zeta functions associated to
  nilpotent algebras of endomorphisms}

\date{\today} \author{Christopher Voll} \address{Fakult\"at f\"ur
  Mathematik, Universit\"at Bielefeld\\ Postfach 100131\\ D-33501
  Bielefeld\\Germany} \email{C.Voll.98@cantab.net}

%  Igusa's local zeta function, $p$-adic integration, local functional
%  equations, Kirillov theory} 
% 11M41 other Dirichlet series and zeta functions
% 20E07 subgroup growth
% 11S40 Zeta functions and $L$-functions
% 16W20 Automorphisms and endomorphisms
%\makeindex 

\begin{abstract} 

  We give a sufficient criterion for generic local functional
  equations for submodule zeta functions associated to nilpotent
  algebras of endomorphisms defined over number fields. This allows
  us, in particular, to prove various conjectures on such functional
  equations for ideal zeta functions of nilpotent Lie lattices. Via
  the Mal'cev correspondence, these results have corollaries
  pertaining to zeta functions enumerating normal subgroups of finite
  index in finitely generated nilpotent groups, most notably
  \revone{finitely generated free nilpotent groups of any given
    class}.

\end{abstract}
\maketitle

\hfill 
-- In memory of Anton Evseev. \hfill

\thispagestyle{empty}
\section{Introduction}

\subsection{Submodule zeta functions for nilpotent associative
  algebras of endomorphisms} \label{subsec:intro} Let $R$ be the ring
of integers $\mcO$ of a number field or \revone{a compact discrete
  valuation ring, such as the completion $\mcO_{\mfp}$ of such a ring
  $\mcO$ at a nonzero prime ideal $\mfp$ of $\mcO$ (in characteristic
  zero) or a formal power series ring $(\mcO/\mfp) \llbracket T
  \rrbracket$ (in characteristic $p>0$)}. Let $\mcL$ be a free
$R$-module of finite rank $n$ and $\mcE$ be a \revtwo{(not necessarily
  unital)} subalgebra of the associative $R$-algebra $\End_R(\mcL)$.
For $m\in\N$, let
%% excised: H \textup{ is  $R$-submodule of $\mcL$},
$$a_m(\mcE\car \mcL) = \#\left\{ H\leq \mcL \mid \; H \textup{ is an
  $(\mcE + R\, \textup{id}_{\mcL})$-submodule of $\mcL$ with $|\mcL:H| = m$} \right\}.$$ We
define the \emph{submodule zeta function of $\mcE$ acting on $\mcL$}
as the formal Dirichlet generating function
$$\zeta_{\mcE\car \mcL}(s) = \sum_{m=1}^\infty a_m(\mcE\car
\mcL)m^{-s},$$ where $s$ is a complex variable;
cf.\ \cite[Definition~2.1~(ii)]{Rossmann/15}. The submodule zeta
function $\zeta_{\mcE\car \mcL}(s)$ may be viewed as a (non-unital)
analogue of Solomon's zeta function; see \cite{Solomon/77}.  Assume
now that $R$ is the ring of integers $\Gri$ of a number field. Then
$\zeta_{\mcE\car \mcL}(s)$ satisfies the Euler product
\begin{equation}\label{equ:euler}
\zeta_{\mcE\car \mcL}(s) = \prod_{\mfp} \zeta_{\mcE(\Gri_{\mfp})\car \mcL(\Gri_{\mfp})}(s),
\end{equation}
where the product ranges over the nonzero prime ideals of $\mcO$ and
$\mcE(\Gri_{\mfp}) := \mcE \otimes_\mcO \mcO_\mfp$ and
$\mcL(\Gri_{\mfp}) := \mcL \otimes_\mcO \mcO_{\mfp}$, regarded as an
$\Gri_{\mfp}$-algebra and $\Gri_{\mfp}$-module, respectively. It
follows from well-known results expressing counting functions such as
$\zeta_{\mcE(\Gri_{\mfp})\car \mcL(\Gri_{\mfp})}(s)$ in terms of
$\mfp$-adic integrals that each of the Euler factors is a rational
function in $|\Gri:\mfp|^{-s}$; see, for instance, \cite{GSS/88} for
the case $\Gri=\Z$ \revone{and, taken together,
  \cite{Igusa/74,duSG/00} for the general case}.

Assume now that $\mcE$ is nilpotent. The main objective of this paper
is to establish in this case, under suitable conditions, functional
equations for $\mfp$-adic submodule zeta functions occurring as
generic factors in Euler products of the form~\eqref{equ:euler}; see
Theorem~\ref{thm:main}.

Prominent examples of submodule zeta functions of nilpotent
associative algebras of endomorphisms are ideal zeta functions of
nilpotent Lie lattices, which we now recall. Let $\mcL$ be an
$\Gri$-\emph{Lie lattice}, i.e.\ a free and finitely generated
$\Gri$-module of finite rank $n$ equipped with an antisymmetric
bi-additive form (or ``Lie bracket'') satisfying the Jacobi
identity. By a \emph{Lie ring} we mean a $\Z$-Lie lattice. For
$m\in\N$, we write
$a_m(\mcL) = \#\{ H \nl_\Gri \mcL \mid |\mcL:H|=m\}$ for the number of
$\Gri$-ideals of $\mcL$ of index $m$ in~$\mcL$. The \emph{ideal zeta
  function of $\mcL$} is the Dirichlet generating
series $$\zeta^{\nl}_{\mcL}(s) = \sum_{m=1}^\infty a_m(\mcL) m^{-s};$$
cf.\ \cite[Section~3]{GSS/88}. It fits into the setup from above by
considering the associative subalgebra $\mcE \subseteq\End_\Gri(\mcL)$
generated by $\ad(\mcL)$; clearly $a_m(\mcL) = a_m(\mcE \car
\mcL)$. The Euler product~\eqref{equ:euler} takes the form
$$\zeta^{\nl}_{\mcL}(s) =
\prod_{\mfp}\zeta^{\nl}_{\mcL(\Gri_{\mfp})}(s),$$
where, for each prime ideal $\mfp$, the Euler factor
$\zeta^{\nl}_{\mcL(\Gri_{\mfp})}(s)$ enumerates the
$\Gri_{\mfp}$-ideals of $\mcL(\Gri_{\mfp})$ of finite additive index
in~$\mcL(\Gri_{\mfp})$.

Returning to general nilpotent associative algebras of endomorphisms
$\mcE\subseteq \End_{\Gri}(\mcL)$ we define $\Gri$-submodules $Z_i$ of
$\mcL$, for $i\in\N_0$, by setting $Z_0=\{0\}$ and
$$Z_{i+1}/Z_i = \Cent_{\mcE}(\mcL/Z_i):= \{x + Z_i \in \mcL/Z_i \mid
x\mcE\subseteq Z_i\}$$ for $i>0$. As $\mcE$ is nilpotent there
exists $i\in\N$ such that~$Z_i=\mcL$; cf.\ \cite[Chapter~2,
  Section~II]{Jacobson/79}. We set $$c=c(\mcL,\mcE)= \min\{i\in\N_0
\mid Z_i = \mcL\}.$$ (If $\mcL$ is a nilpotent $\Gri$-Lie lattice and
$\mcE$ is the associative subalgebra generated by $\ad(\mcL)$, then
$(Z_i)_{i=0}^c$ is just the upper central series of $\mcL$ and $c$ is
the nilpotency class of~$\mcL$.)

In this paper, we consider pairs $(\mcL,\mcE)$ satisfying the
following assumption.

\begin{assumption}\label{ass}
  There exist free $\Gri$-submodules $\mcL_1,\dots,\mcL_{c}$ of
  $\mcL$ such that
\begin{equation*}%\label{equ:cent}
Z_i = \bigoplus_{j > c-i}\mcL_j
\end{equation*}
for $i=1,\dots,c$. Note that, in particular,
\begin{equation*}%\label{equ:direct.sum}
  \mcL  = \mcL_1 \oplus \dots \oplus \mcL_c
\end{equation*}
(direct sum of $\Gri$-modules). \revone{We also set
  $\mcL_{0} = \mcL_{c+1}=\{0\}$.}
\end{assumption}

\begin{remark}\label{rem:ass}
  Assumption~\ref{ass} is only made for notational convenience. It is
  automatically satisfied if $\Gri$ is a unique factorization domain
  (e.g.\ $\Z$). In general, the ``centralizers'' $Z_j$ will be
  isolated in $\mcL$ (viz.\ the factor modules $\mcL / Z_j$ will be
  torsion-free), but may not allow complements. This may be mitigated
  by localizing $\mcL$ at a finite set of prime ideals of $\Gri$ or --
  by the general theory of finitely generated modules over Dedekind
  domains -- by passing to a suitable finite index $\Gri$-submodule of
  $\mcL$; cf., for instance, the discussion in
  \cite[Section~2.3]{StasinskiVoll/14}. In any case, only finitely
  many of the Euler factors in~\eqref{equ:euler} are affected. As we
  are only looking to prove results for all but finitely many of
  these, making Assumption~\ref{ass} means no loss of generality.
\end{remark}

Our main results concern local submodule zeta functions associated to
general nilpotent algebras $\mcE$ which satisfy the following
condition.
\begin{condition}\label{con}  
  The nilpotent associative algebra $\mcE\subseteq \End_{\Gri}(\mcL)$
  is generated by elements $c_1,\dots,c_d$ such that, for all
  $k=1,\dots,d$ and $j=1,\dots,c$,
\begin{equation}\label{equ:shift}
 \mcL_j\revone{c_k} \subseteq \mcL_{j+1}.
\end{equation}
\end{condition}

For a matrix version of this condition, see Condition~\ref{con2}.

Given a nonzero prime ideal $\mfp$ of $\mcO$ we write $q$ for the
cardinality of the residue field~$\mcO/\mfp$. The paper's main result
establishes, in particular, functional equations upon inversion of $q$
for all but finitely many of the Euler factors in \eqref{equ:euler} in
case that $(\mcL,\mcE)$ satisfies Condition~\ref{con}. For
$i\in\{0,1,\dots,c\}$ we write 
$$N_i = \rk_{\Gri} \bigoplus_{j \leq c-i}\mcL_j =
\rk_{\Gri}(\mcL/Z_i),$$ noting that $N_0=n = \rk_{\Gri}\mcL$
and~$N_c=0$. Throughout this paper, by a finite extension of a local
ring of the form $\Gri_{\mfp}$ we mean the ring of integers of a
finite extension of the local field of fractions of~$\Gri_{\mfp}$.

\begin{theorem}\label{thm:main}
  Assume that $\mcE\subseteq \End_{\Gri}(\mcL)$ satisfies
Condition~\ref{con}.  Then, for almost all prime ideals $\mfp$ of
$\mcO$ and all finite extensions $\Lri$ of $\Gri_\mfp$, with residue
field cardinality $q^f$, say, the following functional equation holds:
\begin{equation}\label{equ:funeq.submodule}
  \left. \zeta_{\mcE(\Lri) \car \mcL(\Lri)}(s)\right|_{q \rarr
    q^{-1}} = (-1)^n
  q^{f\left(\binom{n}{2}-s\left(\sum_{i=0}^{c-1}N_i\right)\right)}\zeta_{\mcE(\Lri)
    \car \mcL(\Lri)}(s).
\end{equation}
\end{theorem}

\revone{Whilst the local submodule zeta functions encountered in Euler
  products such as \eqref{equ:euler} are defined in terms of local
  rings of characteristic zero, they have counterparts in sufficiently
  large positive characteristic~$p$. More precisely, expressing the
  relevant submodule zeta functions as $\mfp$-adic integrals -- be it
  cone integrals in the vain of \cite{duSG/00} ((cf.\
  \cite[Theorem~5.16]{Rossmann/15})) or the ones deployed in
  Section~\ref{subsec:proof.main.thm} -- and employing a version of
  the transfer principle for such integrals (for instance
  \cite[Theorem~9.2.4]{CluckersLoeser/10}), yields the following: for
  almost all $\mfp$ and all finite extensions $\Lri$ of $\Gri_{\mfp}$,
  with maximal ideal $\mfP$, say, setting
  $\mcE((\Lri/\mfP)\llbracket T \rrbracket) =
  \mcE\otimes_{\Gri}(\Lri/\mfP)\llbracket T \rrbracket$
  and
  $\mcL((\Lri/\mfP)\llbracket T \rrbracket) =
  \mcL\otimes_{\Gri}(\Lri/\mfP)\llbracket T \rrbracket$, one has
\begin{equation}\label{equ:transfer}
\zeta_{\mcE(\Lri) \car \mcL(\Lri)}(s) = \zeta_{\mcE((\Lri/\mfP)\llbracket T \rrbracket) \car \mcL((\Lri/\mfP)\llbracket T \rrbracket)}(s).
\end{equation}
The following is thus an immediate consequence of
Theorem~\ref{thm:main}.

\begin{corollary}\label{cor:char.p}
  For almost all prime ideals $\mfp$ of $\mcO$ and all finite
  extensions $\Lri$ of $\Gri_\mfp$, with maximal ideal $\mfP$ and
  residue field cardinality $|\Lri/\mfP|=q^f$, say, the following functional
  equation holds:
\begin{equation*}%\label{equ:funeq.submodule}
  \left. \zeta_{\mcE((\Lri/\mfP)\llbracket T \rrbracket) \car \mcL((\Lri/\mfP)\llbracket T \rrbracket)}(s)\right|_{q \rarr
    q^{-1}} = (-1)^n
  q^{f\left(\binom{n}{2}-s\left(\sum_{i=0}^{c-1}N_i\right)\right)}\zeta_{\mcE((\Lri/\mfP)\llbracket T \rrbracket) \car \mcL((\Lri/\mfP)\llbracket T \rrbracket)}(s).
\end{equation*}
\end{corollary} 

(We are grateful to an anonymous referee for asking
about characteristic $p$ and to Raf Cluckers for directing us to
modern versions of the transfer principle.)

\begin{example}\label{exa:c=1}
  For $c=1$, Condition~\ref{con} is trivially satisfied as
  $\mcE = \{0\}$. Theorem~\ref{thm:main} and
  Corollary~\ref{cor:char.p} follow, in this case, by inspection of
  the classical formula, valid for any compact discrete valuation ring
  $\lri$ with residue field cardinality $q$, say,
\begin{equation}\label{equ:c=1}
  \zeta_{\{0\}\car \lri^n}(s) = \frac{1}{(1-q^{-s})(1-q^{1-s})\dots(1-q^{n-1-s})}
\end{equation}
for the Dirichlet generating series enumerating all finite index
submodules of $\lri^n$; see, for instance,
\cite[Proposition~1.1]{GSS/88}.
\end{example}

\begin{remark}
  Whilst we cannot make any quantitative statements on the finite set
  of prime ideals of $\Gri$ to be excluded in Theorem~\ref{thm:main},
  there are examples illustrating that this set of ``bad primes'' is
  not empty in general. We note that the assertion of
  Theorem~\ref{thm:main} is uniform under global base extension in the
  sense that the possible exceptions to the Euler factors' functional
  equations are determined already by $(\mcL,\mcE)$.

  The necessity to disregard finitely many residue characteristics is
  one reason why, Corollary~\ref{cor:char.p} notwithstanding, our
  exposition will for the most part concentrate on local submodule
  zeta functions in characteristic zero. Another motivation to focus
  on submodule zeta functions defined over number fields (as opposed,
  say, to global function fields) comes from the group-theoretic
  application to ideal zeta functions of finitely generated nilpotent
  groups; cf.\ Section~\ref{subsec:appI}.
\end{remark}}
\revone{\begin{remark} 
    In the special case of ideal zeta functions associated to (not
    necessarily nilpotent) Lie rings, viz.\ $\Z$-Lie lattices, some of
    the identities \eqref{equ:transfer} were established in
    \cite[Theorem~5.4]{duS_forum/00}. Indeed, this result shows that
    the ideal zeta functions $\zeta^{\nl}_{\mcL(\Zp)}(s)$ and
    $\zeta^{\nl}_{\mcL(\Fp\llbracket T \rrbracket)}(s)$ of the $\Zp$-
    resp.\ $\Fp\llbracket T\rrbracket$-algebras obtained by tensoring
    a given Lie ring $\mcL$ with $\Zp$ resp.\
    $\Fp\llbracket T \rrbracket$ coincide for large $p$. The theorem
    rests on \cite[Proposition~5.7]{duS_forum/00}, an instance of the
    Ax-Kochen-Ershov principle, attributed in \cite{duS_forum/00} to a
    remark in \cite{Macintyre/90}, which establishes identities
    between certain definable integrals over $\Zp$ resp.\
    $\Fp\llbracket T \rrbracket$ for large $p$.  Neither
    \cite{duS_forum/00} nor \cite{Macintyre/90}, however, discuss the
    validity of the established identities after local ring
    extensions, let alone in a uniform manner.

    The transfer principle as established by Cluckers-Loeser in
    \cite{CluckersLoeser/10} vastly generalizes earlier work,
    including the Ax-Kochen-Ershov principle. It allows, in
    particular, for the treatment of more complex classes of (motivic
    and) $\mfp$-integrals but also -- and most relevantly for the
    applications in the current paper --, cross-characteristic
    comparisons which are valid uniformly for local ring extensions,
    provided the residue characteristic is sufficiently large. This
    extra dimension of uniformity under local ring extensions offered
    by modern versions of the transfer principle is reflected in
    equations~\eqref{equ:transfer}.
 \end{remark}}

\begin{remark}\label{rem:invert}
  In general, the operation $q\rarr q^{-1}$ means the following. Given
  a finite extension $\Lri$ of $\Gri_{\mfp}$, write $\mfP$ for the
  maximal ideal of $\Lri$. Thus $|\Lri / \mfP| = q^f$. Our proof of
  Theorem~\ref{thm:main} will show that, excluding finitely many
  $\mfp$, the zeta function $\zeta_{\mcE(\Lri) \car \mcL(\Lri)}(s)$
  may be written as a finite sum of the form
\begin{equation}\label{equ:denef}
\sum_{i=1}^M |\overline{V_i}(\Lri/\mfP)| W_i(q^f,q^{-fs}),
\end{equation}
where the $|\overline{V_i}(\Lri/\mfP)|$ denote the numbers of
$\Lri/\mfP$-rational points of the reductions modulo $\mfp$ of
$\Gri$-defined smooth projective algebraic varieties $V_i$ and
rational functions $W_i(X,Y)\in\Q(X,Y)$.  By the Weil conjectures, the
numbers $|\overline{V_i}(\Lri/\mfP)|$ may be written as alternating
sums of Frobenius eigenvalues. By $q\rarr q^{-1}$ we mean the
operation of inverting these eigenvalues
(\revone{cf.~\cite[eqs.\ (4.10) and (4.11)]{AKOVI/13}}) and, of
course, evaluating $W_i$ at $(q^{-f},q^{fs})$. If the reductions
$\overline{V_i}$ are smooth (i.e.\ the $V_i$ have good reduction
modulo~$\mfp$), then Poincar\'e duality for \'etale cohomology entails
symmetries among the Frobenius eigenvalues which imply that
$|\overline{V_i}(\Lri/\mfP)|_{q\rarr q^{-1}} = q^{-f\dim V_i}
|\overline{V_i}(\Lri/\mfP)|$.  In the special case that
$|\overline{V_i}(\Lri/\mfP)|$ is given by a polynomial
$F_{V_i}\in\Z[X]$ in the residue field cardinality~$q^f$, this amounts
to the palindromic symmetry $F_{V_i}(X^{-1}) = X^{-\dim
  V_i}F_{V_i}(X)$. The rational functions $W_i$ admit a common
denominator of the form $\prod_{j=1}^r(1-X^{a_j}Y^{b_j})$ for
$a_j\in\N_0$, $b_j\in\N$, $j\in\{1,\dots,r \}$. That the functional
equation~\eqref{equ:funeq.submodule} does not depend on the chosen
formula \eqref{equ:denef} is shown in
\cite[Section~4]{Rossmann/15a}. By \cite[Corollary~4.2]{Rossmann/15a},
it suffices to show \eqref{equ:funeq.submodule} for $\Lri=\Gri_{\mfp}$
for all but finitely many $\mfp$.
\end{remark}

\begin{remark}\label{rem:con.sat}
  Condition~\ref{con} is satisfied if $\mcE$ is cyclic (i.e.\ one may
  choose $d=1$; cf.\ Section~\ref{subsec:cyclic}) or if $\mcE^2=0$
  (i.e.\ $c \leq 2$). Moreover, it is stable under taking direct
  products and (certain) central quotients: If
  $\mcE_1\subseteq \End_{\Gri}(\mcL_1)$ and
  $\mcE_2\subseteq \End_{\Gri}(\mcL_2)$ both satisfy
  Condition~\ref{con}, then so does
  $\mcE_1\times \mcE_2 \subseteq \End_{\Gri}(\mcL_1\oplus \mcL_2)$.
  If $\mcE\subseteq \End_{\Gri}(\mcL)$ satisfies Condition~\ref{con}
  and $M \leq Z_1 = \mcL_c$ is an isolated central $\Gri$-submodule
  admitting a complement in $\mcL_c$, then the induced algebra of
  endomorphisms $\overline{\mcE}\subseteq \End_{\Gri}(\mcL/M)$ also
  satisfies Condition~\ref{con}. As in Remark~\ref{ass}, the condition
  that $M$ is isolated and admits a complement may be dropped at the
  cost of disregarding finitely many (further) prime ideals $\mfp$
  of~$\Gri$.
\end{remark}

In contrast to Assumption~\ref{ass}, Condition~\ref{con} \emph{does}
delineate an interesting property. In Section~\ref{sec:nil} we
discuss, along with several examples and applications of
Theorem~\ref{thm:main} in the context of ideal zeta functions of
nilpotent Lie lattices, a number of known examples of such lattices
whose generic local ideal zeta functions do \emph{not} satisfy the
kind of functional equations established by
Theorem~\ref{thm:main}. \revone{We also comment in
  Section~\ref{sec:nil} on connections between Theorem~\ref{thm:main}
  and related work by du Sautoy and Woodward.}  To our knowledge, in
all cases of ideal zeta functions of nilpotent Lie lattices which are
known to satisfy generic local functional equations,
Condition~\ref{con} is satisfied, supporting the speculation that it
may actually be necessary for such functional equations. \revone{An
  analogy with Igusa's local zeta function, however, suggests a
  caveat; cf.\ Section~\ref{subsec:ilzf}.

  Indeed, }Theorem~\ref{thm:main} may be viewed as an analogue of the
functional equation for the generic Igusa local zeta functions
associated to a {homogeneous} polynomial over
$\Gri$\revone{;~\cite{DenefMeuser/91}}. In the light of this analogy,
Condition~\ref{con} may be viewed as a ``homogeneity condition'' for
nilpotent algebras of endomorphisms. For a further discussion of the
connection with Igusa's local zeta function, necessary vs.\ sufficient
conditions for local functional equations for submodule zeta
functions, and potential interpretations of the left-hand side of
\eqref{equ:funeq.submodule}, see Section~\ref{sec:nec.vs.suff}.

The proof of Theorem~\ref{thm:main} is given in
Section~\ref{subsec:proof.main.thm}.

\subsection{Applications I: normal zeta functions of finitely generated nilpotent groups}\label{subsec:appI}

Results such as Theorem~\ref{thm:main} about ideal zeta functions of
nilpotent Lie lattices have corollaries pertaining to {normal}
subgroup zeta functions of finitely generated nilpotent groups,
enumerating the groups' {normal} subgroups of finite index. Indeed, by
the Mal'cev correspondence, for every finitely generated torsion-free
nilpotent group $G$ there exists a nilpotent Lie ring \revone{(viz.\
  $\Z$-Lie lattice)} $\mcL$ such that, for almost all primes~$p$, the
local ideal zeta function $\zeta^{\nl}_{\mcL(\Zp)}(s)$ coincides with
the local normal subgroup zeta function
$$\zeta^{\nl}_{G,p}(s) = \sum_{H \nl_p G} |G:H|^{-s}$$ of $G$ at~$p$,
enumerating normal subgroups of $G$ of $p$-power index in $G$; see
\cite[Theorem~4.1]{GSS/88}. Moreover, every nilpotent Lie ring arises
in this way.

Applications of Theorem~\ref{thm:main} to ideal zeta functions of
nilpotent Lie lattices are discussed in Section~\ref{sec:nil}. Via the
Mal'cev correspondence, all of them have analogues for normal subgroup
zeta functions of finitely generated nilpotent groups. We only spell
out the following corollary of Theorem~\ref{thm:free} on free
nilpotent Lie lattices. Let $c,d\in\N$ and $F_{c,d}$ be the free
class-$c$-nilpotent group on $d$ generators. For $i\in
\{0,1,\dots,c\}$, set
\begin{equation}\label{equ:mobius}N_i %= \rk_{\Z}\left( \mff_{c,d} /
                                      %\gamma_{c-i+1}(\mff_{c,d})\right)
= \sum_{1 \leq j\leq c-i} \frac{1}{j}\sum_{k|j}\mu(k)d^{j/k},
\end{equation} where $\mu$
is the M\"obius function. This well-known ``Witt formula'' gives the
Hirsch lengths of the quotients of $F_{c,d}$ by the terms of the upper
(or, equivalently, lower) central series. The numbers $N_i$ are also equal to the
$\Z$-ranks of the quotients of the free class-$c$-nilpotent Lie ring on
$d$ generators $\mff_{c,d}$ by the terms of the upper central series;
cf.\ Section~\ref{subsec:free.Lie.ring}. Note that $\mff_{c,d}$ is the
nilpotent Lie ring associated to the nilpotent group $F_{c,d}$ by the Mal'cev
correspondence. Our Theorem~\ref{thm:free} on the generic local ideal
zeta functions of these Lie rings has the following consequence.

\begin{corollary}\label{cor:free}
For almost all primes $p$, the following functional equation holds:
\begin{equation}\label{equ:funeq.free}
  \left.\zeta^{\nl}_{F_{c,d},p}(s) \right|_{p \rarr p^{-1}} =
  (-1)^{N_0}p^{\binom{N_0}{2}-s\left(\sum_{i=0}^{c-1}N_i\right)}\zeta^{\nl}_{F_{c,d},p}(s).
\end{equation}
\end{corollary}

\revone{For $c>2$ we have, in general, no means of quantifying the
  respective finite sets of primes for which \eqref{equ:funeq.free}
  fails to hold. For $c\leq 2$, they are empty: cf.\ \eqref{equ:c=1}
  for $c=1$ and \cite[Theorem~3]{Voll/05a} for $c=2$.}

\subsection{Applications II: degree in $q^{-fs}$ and behaviour at
  $s=-\infty$} 

Let $(\mcL,\mcE)$ be \revone{as in Section~\ref{subsec:intro}} with
$\mcE$ nilpotent, not necessarily satisfying Condition~\ref{con}. The
following definition is analogous to the concept of \emph{uniform
  representability} of families of local zeta functions developed in
\cite[Section~2.3]{Rossmann/15a}; see also
\cite[\S~1.2.4]{duSWoodward/08}.

\begin{definition}\label{def:uni}
  We call the pair $(\mcL,\mcE)$ \emph{almost uniform} if there exists
  a rational function $W\in\Q(X,Y)$ such that, for almost all prime
  ideals $\mfp$ of $\Gri$ and all finite extensions $\Lri$ of
  $\Gri_{\mfp}$, with residue field cardinality $q^f$, say,
  $\zeta_{\mcE(\Lri)\car \mcL(\Lri)}(s) = W(q^f,q^{-fs})$. By abuse of
  notation we also call $\zeta_{\mcE\car \mcL}(s)$ \emph{almost
    uniform} in this case.
\end{definition}

Recall that, for general reasons, for all $\mfp$ and all $\Lri$, the
local submodule zeta function $\zeta_{\mcE(\Lri)\car\mcL(\Lri)}(s)$ is
a rational function in $q^{-fs}$. The \emph{degree} of a rational
function $W=P/Q\in\Q(Z)$ is $\deg_Z W = \deg_Z P - \deg_Z Q$.

\begin{conjecture}\label{con:deg}
  For almost all prime ideals $\mfp$ of $\Gri$ and all finite
  extensions $\Lri$ of $\Gri_{\mfp}$, with residue field cardinality $q^f$,
  say,
\begin{align}
  \deg_{q^{-fs}}\left( \zeta_{\mcE(\Lri)\car \mcL(\Lri)}(s) \right) &
  = - \sum_{i=0}^{c-1}N_i,\label{equ:deg.t}\\ \lim_{s\rarr - \infty}
  \left( q^{-fs\sum_{i=0}^{c-1}N_i}\zeta_{\mcE(\Lri)\car
    \mcL(\Lri)}(s) \right) &=
  (-1)^nq^{-f\binom{n}{2}}.\label{equ:deg.q}
\end{align}
If $(\mcL,\mcE)$ is almost uniform, say $\zeta_{\mcE(\Lri)\car
  \mcL(\Lri)}(s) = W(q^f,q^{-fs})$ for almost all $\mfp$ and all
$\Lri$ for some $W\in\Q(X,Y)$, then $\deg_X W = -\binom{n}{2}$.
\end{conjecture}

Informally speaking, equation \eqref{equ:deg.q} pins down the quotient
of the leading coefficients of the polynomials in $q^{-fs}$ giving
numerator and denominator of the rational function
$\zeta_{\mcE(\Lri)\car \mcL(\Lri)}(s)\in\Q(q^{-fs})$. The functional
equation \eqref{equ:funeq.submodule} established in
Theorem~\ref{thm:main} allows us to confirm this conjecture in the
case that Condition~\ref{con} is satisfied. The following is proven in
Section~\ref{subsec:cor}.

\begin{corollary}\label{cor}
  Conjecture~\ref{con:deg} holds if $(\mcL,\mcE)$ satisfies
  Condition~\ref{con}.
\end{corollary}

Further evidence for Conjecture~\ref{con:deg} is provided by the
numerous examples of ideal zeta functions of nilpotent Lie rings in
\cite{duSWoodward/08} which do not satisfy generic local functional
equations; cf.\ Section~\ref{sec:nil}.

Recall that if the degree $\deg_YW$ of a rational \emph{generating}
function $W= P/Q \in\Q(Y)$ is nonpositive, then $\deg_Y Q$ is the
length of a shortest linear recurrence relation satisfied by the
coefficients of $W$ when expanded as a power series in $Y$;
cf.\ \cite[Theorem~4.1.1]{Stanley/97}. Equation \eqref{equ:deg.t} thus
yields a lower bound on the length of such a linear recurrence
relation. Determining this shortest length precisely seems to be a
challenging problem.

\begin{remark}\label{rem:tobias.t=1}
  Our conjecture \eqref{equ:deg.q} on the behaviour of
  $\zeta_{\mcE(\Lri) \car \mcL(\Lri)}(s)$ at $s=-\infty$ may be
  compared with the conjectural behaviour at $s=0$. In
  \cite[Conjecture~IV ($\mfP$-adic form)]{Rossmann/15} Rossmann
  conjectures that, for all (!)~$\mfp$ and all $\Lri$,
$$\left.(1-q^{-fs})\zeta_{\mcE(\Lri) \car \mcL(\Lri)}(s)\right|_{s=0} = \frac{1}{(1-q^f)(1-q^{2f})\dots(1-q^{(n-1)f})};$$
cf.\ also \cite[Section~8.3]{Rossmann/15}.
\end{remark}

\subsection{Context and related work -- zeta functions of groups, rings, and algebras}
Our proof of Theorem~\ref{thm:main}, presented in
Section~\ref{subsec:proof.main.thm}, proceeds by adapting the
$\mfp$-adic machinery developed in \cite{Voll/10}. There, this
technique is applied to establish generic local functional equations
for a range of zeta functions of groups and rings. The most general of
these applications is to subring zeta functions of arbitrary rings of
finite additive rank, i.e.\ finitely generated abelian groups with
some bi-additive multiplicative structure; see
\cite[Theorem~A]{Voll/10}.  Via the Mal'cev correspondence, this
translates into results for the generic Euler factors of the subgroup
zeta functions of finitely generated nilpotent groups, i.e.\ Dirichlet
generating series enumerating all finite index subgroups of such a
group; see~\cite[Corollary~1.1]{Voll/10}.  In
\cite[Theorem~C]{Voll/10} we prove functional equations for generic
local ideal zeta functions of nilpotent Lie rings of class $2$ (or,
equivalently, generic local normal zeta functions of finitely
generated nilpotent groups of class $2$). Theorem~\ref{thm:main}
generalizes this result; cf.\ \revone{Remark~\ref{rem:con.sat}}.

\revone{Numerous examples of ideal zeta functions of nilpotent Lie
  rings have been computed by various people; see, for instance,
  \cite[Section~2]{duSWoodward/08} for a substantial list. For a large
  number of submodule zeta functions associated to nilpotent algebras
  of endomorphisms see \cite{Rossmann/16} and the database in the
  computer algebra package \cite{Rossmannzeta}.  The paper
  \cite{Rossmann/17} gives an explicit formula for the zeta function
  enumerating submodules invariant under a single nilpotent
  endomorphism; cf.\ Section~\ref{subsec:cyclic} for details.}

A variant of the subgroup zeta function of a finitely generated
nilpotent group $G$ is its \emph{pro-isomorphic zeta function}
$\zeta^{\wedge}_G(s)$, enumerating the finite index subgroups of $G$
whose profinite completions are isomorphic to the profinite completion
of~$G$. These zeta functions, too, enjoy an Euler product
$\zeta^{\wedge}_G(s) = \prod_{p \textrm{
    prime}}\zeta^{\wedge}_{G,p}(s)$ whose factors are rational
functions in~$p^{-s}$. There are numerous examples of groups whose
local pro-isomorphic zeta functions satisfy functional equations akin
to~\eqref{equ:funeq.submodule} (see \cite{duSLubotzky/96, Berman/11}
and \cite{BermanKlopschOnn/15, BermanKlopschOnn/15a}), but also an
example showing that this symmetry phenomenon is not universal for
pro-isomorphic zeta functions
(see~\cite{BermanKlopsch/15}). Formulating necessary and sufficient
criteria for generic local functional equations in this context
remains an interesting challenge.

In \cite{Woodward/08}, Woodward computed the ideal zeta functions of
the full upper triangular $n\times n$-matrices over~$\Z$, as well as a
number of combinatorially defined quotients of these algebras. He
gives sufficient criteria for local functional equations, as well as
some examples suggesting that these criteria might be necessary.

Local functional equations akin to \eqref{equ:funeq.submodule} are
also ubiquitous in the theory of representation zeta functions
associated to arithmetic (and related pro-$p$) groups; see, for
instance, \cite{AKOVI/13,StasinskiVoll/14}.

\iffalse{\subsection{Soluble algebras}
The only computations of submodule zeta functions associated to
soluble, nonnilpotent algebras of endomorphisms we are aware of are
those of ideal zeta functions of soluble Lie lattices. 

\subsection{Integral representations of finite
  groups}\label{subsec:solomon} Solomon's work on submodule zeta
functions for semisimple algebras is a classical predecessor and
counterpart of the present paper. Given a finite group $G$ and a
$\Z[G]$-lattice $\mcL$ inside a $\Q[G]$-module $V$, Solomon studied
the submodule zeta function $\zeta_{\Z[G]\car \mcL}(s)$ in
\cite{Solomon/77}. It turns out that these zeta functions are
``almost'', i.e.\ apart from finitely many Euler factors, products of
translates of Dedekind zeta functions of number fields featuring in
the Wedderburn decomposition of the semisimple algebra~$\Q[G]$. In
particular, $\zeta_{\Z[G]\car \mcL}(s)$ is an Euler product, whose
factors are indexed by the rational prime numbers. For almost all $p$,
the relevant Euler factor is a rational function in $p^{-s}$ of a
comparatively simple form, depending only on combinatorial data
globally determined by $V$ and the rational primes' decomposition
types in the various relevant number fields.}\fi

\subsection{Notation}\label{subsec:not}
We write $\N$ for the natural numbers $\{1,2,\dots\}$. Given a subset
$I\subseteq \N$, we write $I_0$ for $I\cup\{0\}$. Given $m,n\in\N_0$,
we set $[m]=\{1,2,\dots,m\}$ and $]m,n]=\{m+1,\dots,n\}$. We write
$\diag\left( \lambda_1^{(e_1)},\dots,\lambda_m^{(e_m)} \right)$ for
the diagonal matrix composed of the matrices $\lambda_1
\Id_{e_1},\dots,\lambda_m\Id_{e_m}$. Given matrices $A_1,\dots,A_n$
with the same number of rows, we denote by $(A_1| \dots | A_n)$ the
matrix obtained by juxtaposition.

Throughout, $p$ is a rational prime, $\Gri$ the ring of integers of a
number field, and $\mfp$ a nonzero prime ideal of $\Gri$. We write
$\lri$ to denote a \revone{compact discrete valuation ring of
  arbitrary characteristic, with maximal ideal $\mfp$, residue field
  $\lri / \mfp$ of cardinality $q$ and characteristic~$p$, and
  $\mfp$-adic valuation~$v$ \revone{(normalized so that every
    uniformizer of $\lri$ has valuation $1$)}. If $\lri$ is of
  characteristic zero, then it is of the form} $\Gri_\mfp$, the
completion of $\Gri$ at a non-zero prime ideal~$\mfp$. Note that these
rings are exactly the finite extensions of the $p$-adic
integers~$\Zp$. In this case we denote by $\Lri$ a finite extension of
$\lri$, with maximal ideal $\mfP$ and residue field cardinality
$|\Lri / \mfP|= q^f$. In other words, $f=f(\Lri,\lri)$ is the relative
degree of inertia. Given a matrix $A =(A_{i,j})\in \Mat_{m,n}(\lri)$
we write $v(A)= \min\{v(A_{i,j}) \mid i\in[m], j\in[n]\}$ for the
minimal valuation of its entries.

The ``Kronecker delta'' $\delta_P$ associated to a property $P$ is
equal to $1$ if $P$ holds and equal to $0$ otherwise.

\section{Proofs of Theorem~\ref{thm:main} and
  Corollary~\ref{cor}}\label{sec:proof}

\subsection{Proof of
  Theorem~\ref{thm:main}}\label{subsec:proof.main.thm}
\subsubsection{Overview of the proof}

Let $\mfp $ be a nonzero prime ideal of $\Gri$. We write $\lri =
\Gri_\mfp$ and $K_\mfp$ for the field of fractions of $\lri$.  We are
looking to establish the functional
equation~\eqref{equ:funeq.submodule} for almost all zeta functions
$\zeta_{\mcE(\lri)\car \mcL(\lri)}(s)$, enumerating \revone{the
  $\lri$-submodules $\Lambda$ of finite index in $\mcL(\lri) \cong
  \lri^n$ which are also $\mcE(\lri)$-submodules}, written $\Lambda
\leq_{\mcE(\lri)}\mcL(\lri)$. Clearly the latter property is really a
property of the integral members of the \emph{homothety class}
$[\Lambda] = \{x \Lambda\mid x\in K_\mfp^*\}$ of $\Lambda$
in~$K_\mfp^n$: either all elements of $[\Lambda]$ are
$\mcE(\lri)$-submodules, or none are. We write
$[\Lambda]\leq_{\mcE(\lri)}\mcL(\lri)$ in the former case and set
$$\SubMod_{\mcE(\lri)} = \{[\Lambda] \mid \Lambda \textup{
  $\lri$-lattice in $K_{\mfp}^n$},
\;[\Lambda]\leq_{\mcE(\lri)}\mcL(\lri)\}.$$ Evidently, every homothety
class $[\Lambda]$ of $\lri$-lattices in $K_{\mfp}^n$ contains a unique
element $\Lambda_{\max}$ which is contained in $\mcL(\lri)$ and
maximal with respect to this property. As the intersection of
$[\Lambda]$ with the set of all full sublattices of $\mcL(\lri)$
equals $\{ \mfp^m \Lambda_{\max} \mid m\in\N_0\}$ it thus suffices --
in principle -- to describe the elements of $\SubMod_{\mcE(\lri)}$ and
to keep track of their maximal integral members' indices
in~$\mcL(\lri)$. Indeed,

\begin{equation}\label{equ:building}
  \zeta_{\mcE(\lri)\car \mcL(\lri)}(s) = \frac{1}{1-q^{-ns}}
  \sum_{[\Lambda] \in \SubMod_{\mcE(\lri)}}
  |\mcL(\lri):\Lambda_{\max}|^{-s}.
\end{equation}

Geometrically, one may view the set of homothety classes of
$\lri$-lattices in $K_\mfp$ as the set of vertices $\mcV_n$ of the
affine Bruhat-Tits building $\Delta(\SL_n(K_\mfp))$; see, for
instance, \cite[Chapter~19]{Garrett/97}. Keeping track of the indices
$|\mcL(\lri):\Lambda_{\max}|$ is easy (cf.~\eqref{equ:index}), so the
problem of computing the right-hand side of \eqref{equ:building} is to
identify $\SubMod_{\mcE(\lri)}$ as a subset of $\mcV_n$.  We give an
informal overview of our way to address this challenge, deferring
details for the time being.

Firstly, we define an equivalence relation $\sim$ on $\mcV_n$ (viz.\
$\delta$-\emph{equivalence}; cf.\ Definition~\ref{def:equiv}) with the
property that -- provided $\mcE\neq \{0\}$, as we may assume without
loss of generality -- each $\sim$-class $\mcC$ is a totally ordered
set naturally isomorphic to $(\Z,\leq)$. Moreover, the sets
$\mcC_{\geq0} := \mcC \cap \SubMod_{\mcE(\lri)}$ correspond, under
these isomorphisms, to $(\N_0,\leq)$.  Setting
\begin{equation}\label{equ:Z.nonneg}
  \Xi_{\mcC_{\geq0}}(s) =  \sum_{[\Lambda]\in\mcC_{\geq0}} |\mcL(\lri):\Lambda_{\max}|^{-s}
\end{equation}
we thus obtain
\begin{equation}\label{equ:zeta.nonneg}
 \zeta_{\mcE(\lri)\car \mcL(\lri)}(s) = \frac{1}{1-q^{-ns}}
\sum_{\mcC \in \mcV_n/\sim} \Xi_{\mcC_{\geq0}}(s).
\end{equation}
Rather than to analyze the functions $\Xi_{\mcC_{\geq0}}(s)$ directly,
we extend the sums~\eqref{equ:Z.nonneg} defining them to the whole
of~$\mcC$ in a way that allows us to recover the original sum
algebraically, uniformly over all $\mcC$. The motivation to consider
these extensions in the first place is that they give rise to a
Dirichlet generating series $A^{\nl}(s)$ satisfying a functional
equation of the desired type; cf.\ \eqref{equ:funeqA}.

To this end we introduce, secondly, judiciously chosen functions
$\wtm_1,m_2:\mcV_n\rarr \N_0$ with the properties
\begin{align*}
[\Lambda]\in \SubMod_{\mcE(\lri)} \Leftrightarrow \wtm_1([\Lambda])=0, \quad
[\Lambda]\in \SubMod_{\mcE(\lri)} \Rightarrow m_2([\Lambda])=0.
\end{align*}
The function $\wtm_1$ may be thought of as measuring a kind of
``distance'' to~$\SubMod_{\mcE(\lri)}$.  We set, for each $\mcC \in
\mcV_n/\sim$,
 $$\Xi_{\mcC}(s) = \sum_{[\Lambda]\in\mcC}
 |\mcL(\lri):\Lambda_{\max}|^{-s} q^{-s\left((c-1) \wtm_1([\Lambda]) -
     m_2([\Lambda])\right)},$$ naturally extending the sum
 \eqref{equ:Z.nonneg} defining $\Xi_{\mcC\geq 0}(s)$. Our choices of
 $\wt{m}_1$ and $m_2$ will ensure that -- up to a power of $q^{-s}$
 depending on $\mcC$ -- the function $\Xi_{\mcC}(s)$ is the sum of two
 geometric progressions, covering respectively the ``nonnegative''
 part $\mcC_{\geq0}$ and the ``negative'' part $\mcC_{<0} := \mcC
 \setminus \mcC_{\geq0}$ of $\mcC$, which only depend on the data
 $(N_i)_{i=0}^{c-1}$ but, crucially, not on~$\mcC$; cf.\
 Corollary~\ref{cor:pos} and Lemma~\ref{lem:m}. This entails that,
 with $$\revtwo{\Xi(s) :=
 \frac{1-q^{-s\sum_{i=1}^{c-1}(n-N_i)}}{1-q^{-s(c-1)n}}},$$ we have
 $\Xi_{\mcC_{\geq0}}(s) = \Xi_{\mcC}(s) \Xi(s)$ for all $\sim$-classes
 $\mcC$.  As indicated above, the point of these constructions is that
 the functions $\wtm_1$ and $m_2$ are defined in such a way that -- at
 least for almost all $\mfp$ -- the function
$$A^{\nl}(s) := \sum_{\mcC \in \mcV_n/\sim} \Xi_{\mcC}(s) =
\sum_{[\Lambda]\in \mcV_n} |\mcL(\lri):\Lambda_{\max}|^{-s}
q^{-s\left((c-1) \wtm_1([\Lambda]) - m_2([\Lambda])\right)}$$ may be
expressed in terms of $\mfp$-adic integrals which fit the templates
provided by~\cite{Voll/10}. In particular, it satisfies, for almost
all $\mfp$, the functional equation~\eqref{equ:funeqA}. Moreover,
\begin{multline*}
  \zeta_{\mcE(\lri)\car \mcL(\lri)}(s) = \frac{1}{1-q^{-ns}}\sum_{\mcC
    \in\mcV_n/\sim}\Xi_{\mcC_{\geq 0}}(s) = \\ \frac{1}{1-q^{-ns}}
  \Xi(s) \sum_{\mcC \in \mcV_n/\sim} \Xi_{\mcC}(s) =
  \frac{1}{1-q^{-ns}} \Xi(s) A^{\nl}(s).
\end{multline*}
The factor $\frac{1}{1-q^{-ns}} \Xi(s)$ trivially satisfies a
functional equation of the required type;
see \eqref{equ:funeqtrivial}. Together with \eqref{equ:funeqA}, this
will prove Theorem~\ref{thm:main}.

\subsubsection{Cocentral bases}
For $i\in[c]_0$ we write $n_i = \rk_\Gri(\mcL_i)$, so that $N_i =
\rk_\Gri(\mcL/Z_i)= \sum_{j\leq c-i}n_j$. An $\Gri$-basis $\bfe =
(e_1,\dots,e_n)$ of $\mcL$ is called \emph{cocentral} if $$Z_i = \la
e_{N_i+1},\dots,e_n\ra_\Gri$$ for all $i\in[c]$. (This terminology
extends the one introduced in \cite[Definition~4.37]{duSWoodward/08}.)
By Assumption~\ref{ass}, cocentral bases clearly
exist. Condition~\ref{con} is equivalent to the following condition.

\begin{condition}\label{con2}
  There exist generators $c_1,\dots,c_d$ of $\mcE$ and a cocentral
  $\Gri$-basis $\bfe$ of $\mcL$ such that, for all $k\in[d]$, the
  matrix $C_k$ representing $c_k$ with respect to $\bfe$ (acting from
  the right on row vectors) has the form $$C_k = \left(C_{k}^{(ij)}
  \right)_{i,j\in[c]}\in\Mat_n(\Gri)$$ for blocks
  $C_{k}^{(ij)}\in\Mat_{n_i,n_j}(\Gri)$ which are zero unless $j=i+1$.
\end{condition}

\subsubsection{Lattices, matrices, and the submodule condition}

Let $\bfe$ be a cocentral $\Gri$-basis of $\mcL$ as in
Condition~\ref{con2}. It yields an $\lri$-basis of $\mcL(\lri)$ which
we also denote by $\bfe$ and which allows us to identify the
$\lri$-module $\mcL(\lri)$ with $\lri^n$ and $\mcE(\lri)$ with a
nilpotent subalgebra of $\Mat_n(\lri)$, generated by matrices
$C_1,\dots,C_d$ representing the $\lri$-linear operators
$c_1,\dots,c_d$.  Set $\Gamma = \GL_n(\lri)$. A full $\lri$-sublattice
$\Lambda$ of $\mcL(\lri)$ may be identified with a coset $\Gamma M$
for a matrix $M\in\GL_n(K_{\mfp})\cap \Mat_n(\lri)$, whose rows encode
the coordinates with respect to~$\bfe$ of a set of generators of
$\Lambda$. Let $\pi\in\mfp$ be a uniformizer of $\lri$. By the
elementary divisor theorem there exist
$I=\{i_1,\dots,i_l\}_<\subseteq[n-1]$, $r_n\in\N_0$, and
$\bfr_{\revone{I}}=(r_{i_1},\dots,r_{i_l})\in\N^l$, all uniquely determined by
$\Lambda$, and $\alpha\in \Gamma$ such that $M=D\alpha^{-1}$, where
\begin{equation}\label{equ:diagonal}
  D = \pi^{r_n} \diag((\pi^{\sum_{\iota\in I}r_\iota})^{(i_1)},(\pi^{\sum_{\iota\in I \setminus\{i_1\}}r_\iota})^{(i_2-i_1)},\dots,(\pi^{r_{i_l}})^{(i_l-i_{l-1})},1^{(n-i_l)})\in\Mat_n(\lri).
\end{equation}
We write $\nu([\Lambda])=(I,\bfr_I)$. Note that $r_n = v(M)$ in the
notation introduced in~Section~\ref{subsec:not}.

The matrix $\alpha\in\Gamma$ is only unique up to right multiplication
by an element of a finite index subgroup $\Gamma_{I,\bfr}$ of
$\Gamma$; see \cite[Section~3.1]{Voll/10} for details. Obviously,
\begin{equation}\label{equ:index}
  |\mcL(\lri):\Lambda| = q^{v(\det D)} = q^{\sum_{\iota\in
      I\cup\{n\}}\iota r_\iota}.
\end{equation}
We call $\Lambda$ \emph{maximal} if $r_n=0$ and denote by
$\Lambda_{\max}$ the unique maximal integral element of $[\Lambda]$.
In the sequel we will often -- and sometimes without explicit
mentioning -- toggle between lattices $\Lambda$ and representing
matrices $M$, extending notation for lattices to matrices representing
them. We write, for instance, $[M]$ for the homothety class
$[\Lambda]$ of the lattice $\Lambda$ determined by $\Gamma M$ and $M
\leq_{\mcE(\lri)} \mcL(\lri)$ if $\Lambda
\leq_{\mcE(\lri)}\mcL(\lri)$.  We also write, given matrices $A$ and
$B$ over $\lri$ with the same number of columns, $A \leq B$ if each
row of $A$ is contained in the $\lri$-row span of $B$.

Define the diagonal matrix
\begin{equation}\label{def:delta}
\delta :=
\diag((\pi^{c-1})^{(n_1)},\dots,(\pi)^{(n_{c-1})},1^{(n_c)})\in\Mat_n(\lri).
\end{equation}
We remark that, up to a scalar factor, $\delta$ represents a map
closely related to the map $\pi_{\mathcal{B}}$ in
\cite[Definition~4.40]{duSWoodward/08}. Note that $\det \delta =
\pi^{\sum_{i=1}^{c-1}N_i}$. The following is a trivial consequence of
Condition~\ref{con2}; for our purposes it is key nevertheless.

\begin{lemma}\label{lem:3}
 If $c>1$, then $\delta C_k \delta^{-1} = \pi C_k$ for all $k\in[d]$.
\end{lemma}

For $i\in[n]$ and $k\in[d]$, write
$(e_i)\revone{c_k}=\sum_{l=1}^n\lambda_{i,k}^l e_l$ for
$\lambda_{i,k}^l\in\Gri$. Then $C_k$ satisfies $\left(
C_k\right)_{r,s}=\lambda^s_{r,k}$ for $r,s\in[n]$.  Let
$\bfY=(Y_1,\dots,Y_n)$ be independent variables and set
$$\mcR(\bfY) = \left( \sum_{l=1}^n \lambda_{i,k}^l
  Y_l\right)_{i,k}\in\Mat_{n\times d}(\Gri[\bfY]).$$ Note that $\mcR=0
\Leftrightarrow \mcE=0 \Leftrightarrow c=1$. In this case,
Theorem~\ref{thm:main} holds (cf.\ Example~\ref{exa:c=1}), so we may
assume that $c>1$. We write $\alpha[i]$ for the $i$-th column of a
matrix~$\alpha\in\Gamma$, so that $\mcR(\alpha[i])\in\Mat_{n\times
  d}(\lri)$.  The following lemma is verified by a trivial
computation.

\begin{lemma}\label{lem:2}
  For all $\alpha\in\Gamma$ and $\Delta\in\Mat_n(\lri)$, and $D$ as in
  \eqref{equ:diagonal},
$$\left( \forall k\in[d]:\; \Delta C_k\alpha\leq D\right) \Leftrightarrow \left( \forall
i\in [n]:\; \Delta \mcR(\alpha[i])\equiv 0 \bmod D_{i,i}\right).$$
\end{lemma}

\begin{proposition}\label{prop:1}
Given $M\in\GL_n(K_{\mfp})\cap \Mat_n(\lri)$, there exists a unique
$\wtm_1=\wtm_1(M)\in\N_0$ such that, for all $m\in\N_0$,
$$M\delta^m\leq_{\mcE(\lri)} \mcL(\lri)\textup{ if and only }m\geq
\wtm_1.$$ In particular, $M \leq_{\mcE(\lri)} \mcL(\lri)$ if and only
if $\wtm_1=0$. Moreover, $\wtm_1 \leq \sum_{\iota\in I}r_{\iota}$.
\end{proposition}

\bp Write $M = D \alpha^{-1}$ as above. Without loss we may assume
$r_n=0$. Using Lemmas~\ref{lem:3} and \ref{lem:2} and the fact that
$D_{i,i}=\pi^{\sum_{i\leq \iota\in I}r_\iota}$ for $i\in[n]$, we obtain
\begin{align*} M\delta^m \leq_{\mcE(\lri)} \mcL(\lri)\\
  \Leftrightarrow \;& \forall k\in [d]: M \delta^m C_k \leq M \delta^m\\
  \Leftrightarrow \; & \forall k\in [d]: D \alpha^{-1}\delta^m C_k \delta^{-m}\alpha \leq D \\
 \Leftrightarrow \; & \forall k\in [d] : \pi^m D \alpha^{-1}C_k \alpha \leq D\\
\Leftrightarrow \; & \forall i\in[n]: \pi^mD\alpha^{-1}\mcR(\alpha[i]) \equiv 0 \bmod D_{i,i}\\
\Leftrightarrow \; & \forall i\in[n]: \pi^m D \alpha^{-1}\mcR(\alpha[i])\pi^{\sum_{i > \iota \in I} r_\iota} \equiv 0 \bmod \pi^{\sum_{\iota\in I} r_\iota}\\
\Leftrightarrow \; & \pi^m D \alpha^{-1} \left( \mcR(\alpha[1]) \mid \dots \mid \mcR(\alpha[n])\right)\cdot \\
& {\diag(1^{(di_1)},(\pi^{r_{i_1}})^{(d(i_2-i_1))},\dots,(\pi^{\sum_{\iota\in I}r_\iota})^{(d(n-i_l))})\equiv 0 \bmod \pi^{\sum_{\iota\in I}r_\iota}.}
\end{align*}
In the last congruence, we may replace $\alpha^{-1}$ by the adjunct
matrix $\alpha^{\textup{adj}}$.  Setting, for $i,r\in[n]$,
\begin{align*}
\mcR_{(i)}(\alpha) &= \alpha^{\textup{adj}}\mcR(\alpha[i]),\\
v^{(1)}_{i,r}(\alpha) &= \min \left\{
v(\mcR_{(\iota)}(\alpha)_{\rho\sigma}) \mid \iota \leq i, \rho \geq r,
\sigma\in[d]\right\},
\end{align*}
and
\begin{equation*}%\label{equ:m1}
  m_1(M) = \min \left\{ \sum_{\iota\in I}r_\iota, \sum_{r \leq \iota
    \in I} r_\iota + \sum_{i > \iota \in I}r_{\iota} +
  v^{(1)}_{i,r}(\alpha) \bigm|  (i,r)\in[n]^2\right\},
\end{equation*}
we may rephrase the above equivalence as follows:
\begin{equation}\label{equ:m1.tilde}
  M\delta^m \leq_{\mcE(\lri)} \mcL(\lri) \Leftrightarrow  m \geq \sum_{\iota\in I}r_{\iota} - m_1(M) =: \wtm_1(M).\qedhere
\end{equation}
  \ep

\begin{definition}\label{def:dist}
  For a lattice $\Lambda$ corresponding to a coset $\Gamma M$, we set
  $\wtm_1([\Lambda]) = \wtm_1(M)$.
\end{definition}

Informally, $\wtm_1([\Lambda])$ is a `distance' in $\mcV_n$ between
$[\Lambda]$ and $\SubMod_{\mcE(\lri)}$.

\subsubsection{$\delta$-equivalence}
Recall the diagonal matrix $\delta$ defined in~\eqref{def:delta}.

\begin{definition}\label{def:equiv}
  Lattice classes $[\Lambda_1],[\Lambda_2]\in \mcV_n$ are called
  $\delta$-\emph{equivalent}, written $[\Lambda_1]\sim[\Lambda_2]$, if
  there exists $m\in\Z$ such that $[\Lambda_1] = [\Lambda_2
    \delta^m]$.
\end{definition}

In the sequel, we use the terms \emph{lattice class} for a homothety
class of lattices and \emph{$\delta$-class} for a $\sim$-equivalence
class of lattice classes.  Proposition~\ref{prop:1} asserts that every
$\delta$-class of lattice classes intersects $\SubMod_{\mcE(\lri)}$
nontrivially. Its proof also shows that, more precisely, in each
$\delta$-class $\mcC$ there is a unique lattice class $[\Lambda_0]$
such that $[\Lambda_0\delta^m] \leq_{\mcE(\lri)} \mcL(\lri)$ if and
only if $m\in\N_0$. We shall say that $[\Lambda_0]$ \emph{generates}
$\mcC_{\geq 0}$ and write $\Lambda_{0,\max}$ for the unique maximal
element of~$[\Lambda_0]$. Setting
\begin{align*}
  \mcC_{\geq0} & = \{[\Lambda_0 \delta^m] \mid m\geq 0\} = \mcC \cap
  \SubMod_{\mcE(\lri)}, \\
  \mcC_{<0} & = \{ [\Lambda_0 \delta^m] \mid m < 0\}= \mcC \setminus
  \mcC_{\geq0},
\end{align*} we obtain a partition $\mcC = \mcC_{\geq 0} \cup
\mcC_{<0}$ and \eqref{equ:zeta.nonneg} holds with
$\Xi_{\mcC_{\geq0}}(s)$ defined as in \eqref{equ:Z.nonneg}.

Note that clearly $v(M) \leq v(M\delta)$ for all
$M\in\GL_n(K_{\mfp})\cap \Mat_n(\lri)$.

\begin{lemma}\label{lem:delta}
  For almost all prime ideals $\mfp$, the following holds for all
  $M\in\GL_n(K_\mfp) \cap \Mat_n(\Gri_\mfp)$: if $M
  \leq_{\mcE(\Gri_{\mfp})}\mcL(\Gri_{\mfp})$, then $v(M) =
  v(M\delta)$.
\end{lemma}

\bp We proceed by induction on $c$, including the case $c=1$, which we
excluded in the previous arguments. Indeed, for this base case the
statement holds trivially (and for all~$\mfp$) as~$\delta = \Id_n$.
Assume thus that $c\geq 2$ and that the induction hypothesis
holds. Given $\mfp$ and~$M$, write $\Lambda$ for the lattice defined
by $M$ and set $\lri = \Gri_{\mfp}$. Without loss we may assume that
$M$ is a block-upper triangular matrix, i.e.\ composed of blocks
$M^{(ij)}\in\Mat_{n_i, n_j}(\lri)$, $i,j\in[c]$, with $M^{(ij)} = 0$
unless $j\geq i$.

\revtwo{The claim is that, for almost all $\mfp$ and all $M$, the
  minimal $\mfp$-adic valuation of the entries of $M$ is equal to that
  of its last block column: if $\pi$ divides one of the block
  matrices~$M^{(ic)}$, $i\in[c]$, then it divides the whole
  matrix~$M$. By induction hypothesis (and excluding finitely
  many~$\mfp$), we may assume that the matrix
  $M' := (M^{(ij)})_{i,j\in[2,c]}\in\Mat_{n-N_{c-1}}(\lri)$, defining
  the lattice $\Lambda \cap Z_{c-1}(\lri)$, where
  $Z_{c-1}(\lri) = Z_{c-1}\otimes_{\Gri}\lri$, has the desired
  property that if $\pi$ divides the last block column of $M'$, viz.\
  one of the block matrices~$M^{(ic)}$, $i\in[2,c]$, then $\pi$
  divides the whole matrix $M'$. So assume that $\pi$ divides the last
  block column of $M$ and thus that $M^{(ij)}\equiv 0 \bmod \pi$ for
  $i\geq 2$ or $j=c$, but that there exists $j\in[c-1]$ such
  that~$v(M^{(1j)})=0$.  Thus one of the first $n_1$ rows of $M$ is
  nonzero modulo $\pi$, defining an element
  $x\in \mcL(\lri)/Z_{1}(\lri)$ which is nonzero modulo~$\pi$.  But
  for all but finitely many $\mfp$ the reduction modulo $\mfp$ of
  $Z_1(\lri)$ is the centraliser of
  $\mcL(\lri/\mfp) = \mcL \otimes_{\lri} \lri/\mfp$. This establishes
  the claim.}  \iffalse [old] The claim is that, for almost all
$\mfp$, the minimal $\mfp$-adic valuation of the entries of $M$ is
equal to that of its last block column: if $\pi$ divides one of the
block matrices~$M^{(ic)}$, $i\in[c]$, then it divides the whole
matrix~$M$. By induction hypothesis, the matrix
$(M^{(ij)})_{i,j\in[2,c]}\in\Mat_{n-N_{c-1}}(\lri)$, defining the
lattice $\Lambda \cap Z_{c-1}(\lri)$, where
$Z_{c-1}(\lri) = Z_{c-1}\otimes_{\Gri}\lri$, has the desired
property. So assume that $\pi$ divides the last block column of $M$
and thus that $M^{(ij)}\equiv 0 \bmod \pi$ for $i\geq 2$ or $j=c$, but
that there exists $j\in[c-1]$ such that~$v(M^{(1j)})=0$. Thus one of
the first $n_1$ rows of $M$ is nonzero modulo $\pi$, defining an
element $x\in \mcL(\lri)/Z_{1}(\lri)$ which is nonzero modulo
$\pi$. But as the reduction modulo $\mfp$ of $Z_1(\lri)$ is, for all
but finitely many~$\mfp$, the centraliser of
$\mcL(\lri/\mfp) = \mcL \otimes_{\lri} \lri/\mfp$, this yields the
desired contradiction. \fi \ep

Assume from now that $\mfp$ satisfies the conclusions of
Lemma~\ref{lem:delta}.

\begin{corollary}\label{cor:pos}
  For every $\mcC \in \mcV_n/\sim$,
$$\Xi_{\mcC_{\geq0}}(s) = |\mcL(\lri):\Lambda_{0,\max}|^{-s}
  \frac{1}{1-q^{-s\sum_{i=1}^{c-1}N_i}}.$$
\end{corollary}

\bp For all $m\in\N_0$ we have $\Lambda_{0,\max}\delta^m = \left(
  \Lambda_{0,\max}\delta^m \right)_{\max}$ by Lemma~\ref{lem:delta}.
Hence
$$|\mcL(\lri):\Lambda_{0,\max}\delta^m| =
|\mcL(\lri):\Lambda_{0,\max}|q^{m\sum_{i=1}^{c-1}N_i}$$ and therefore
\begin{multline*}
  \Xi_{\mcC_{\geq 0}}(s) = \sum_{[\Lambda]\in \mcC_{\geq
      0}}|\mcL(\lri):\Lambda_{\max}|^{-s} = \sum_{m=0}^\infty | \mcL(\lri): \Lambda_{0,\max}\delta^m |^{-s} =\\
  |\mcL(\lri):\Lambda_{0,\max}|^{-s}\sum_{m=0}^\infty
  q^{-s m\sum_{i=1}^{c-1}N_i} = |\mcL(\lri):\Lambda_{0,\max}|^{-s}
  \frac{1}{1-q^{-s\sum_{i=1}^{c-1}N_i}}.\qedhere
\end{multline*} \ep

We set $\wtdelta = \pi^{c-1}\delta^{-1}$ and note that $\det
\wtdelta = \pi^{\sum_{i=1}^{c-1}(n-N_i)}$ and $\mcC_{<0} = \{
[\Lambda_0 \wtdelta^m] \mid m > 0\}$. We seek to describe the `weight
function' $w:\mcV_n\rarr \N_0$ defined by the property that
$w|_{\SubMod_{\mcE(\lri)}}=0$ and, for each $\delta$-class $\mcC$,
\begin{align}
  \Xi_{\mcC_{<0}}(s) &:=
  \sum_{[\Lambda]\in\mcC_{<0}}|\mcL(\lri):\Lambda_{\max}|^{-s}q^{-s n
    w([\Lambda])} =
  \sum_{m=1}^\infty|\mcL(\lri):\Lambda_{0,\max}\wtdelta^m|^{-s}\label{equ:weight}.
\end{align}

Note that then 
$$\Xi_{\mcC_{<0}}(s) = 
  |\mcL(\lri):\Lambda_{0,\max}|^{-s} \sum_{m=1}^\infty q^{-s m
    \sum_{i=1}^{c-1}(n-N_i)} = |\mcL(\lri):\Lambda_{0,\max}|^{-s}
  \frac{q^{-s\sum_{i=1}^{c-1}(n-N_i)}}{1-q^{-s\sum_{i=1}^{c-1}(n-N_i)}}.$$

To obtain this analogue of the formula for $\Xi_{\mcC_{\geq0}}(s)$
established in Corollary~\ref{cor:pos}, we need to take care to define
$w$ judiciously: the point here is that, whilst with
$\Lambda_{0,\max}$ also $\Lambda_{0,\max}\delta^m$ is maximal
(cf.\ Lemma~\ref{lem:delta}), the lattice
$\Lambda_{0,\max}\wtdelta^m$ is not, in general, as the following
example illustrates.

\begin{example}
  Let $\mcL(\lri) = \la x,y,z \mid [x,y]=z, [x,z]=[y,z]=0\ra_\lri$ be
  the Heisenberg $\lri$-Lie lattice. Consider the matrices
$$ M_1 = \left( \begin{matrix}1&&\\&1&\\&&\pi\end{matrix}\right),
\quad M_2 =
\left( \begin{matrix}1&&1\\&1&1\\&&\pi\end{matrix}\right)\in\Mat_3(\lri),$$
encoding sublattices of index $q$ in $\mcL(\lri)\cong \lri^3$ \revone{with
respect to the $\lri$-basis $(x,y,z)$}.  Clearly neither of them are
ideals of $\mcL(\lri)$ (in fact, $\wtm_1(M_1) = \wtm_1(M_2)=1$) and
their homothety classes are $\delta$-inequivalent. Considering
$$ M_1\delta =
\left( \begin{matrix}\pi&&\\&\pi&\\&&\pi\end{matrix}\right)=\pi\Id_3,
\quad M_2\delta =
\left( \begin{matrix}\pi&&1\\&\pi&1\\&&\pi\end{matrix}\right),$$
it is clear that whilst $M_2\delta$ is maximal, $M_1\delta$ is
not. Thus $w([M_1]) = 0$ but $w([M_2])=1$.
\end{example} 

Given  $w$ to the requirements of \eqref{equ:weight}, it
suffices to show that the function
\begin{equation*}%\label{equ:A}
  A^{\nl}(s) := \sum_{[\Lambda]\in \mcV_n}|\mcL(\lri):\Lambda_{\max}|^{-s}
  q^{-snw([\Lambda])}
\end{equation*}
satisfies the functional equation
\begin{equation}\label{equ:funeqA}
\left.A^{\nl}(s) \right|_{q\rarr q^{-1}} =
(-1)^{n-1}q^{\binom{n}{2}}A^{\nl}(s).
\end{equation}
Indeed, $A^{\nl}(s)$ is, by design of $w$, as \revone{a} Dirichlet generating
series with nonnegative coefficients, divisible by
$$\underbrace{\frac{1}{1-q^{-s\sum_{i=1}^{c-1}N_i}}}_{\textup{sum over
    $\mcC_{\geq0}$}} +
\underbrace{\frac{q^{-s\sum_{i=1}^{c-1}(n-N_i)}}{1-q^{-s\sum_{i=1}^{c-1}(n-N_i)}}}_{\textup{sum
    over $\mcC_{<0}$}} =
\frac{1-q^{-s(c-1)n}}{(1-q^{-s\sum_{i=1}^{c-1}N_i})(1-q^{-s\sum_{i=1}^{c-1}(n-N_i)})}.$$
We want, however, to isolate the geometric progression
${(1-q^{-s\sum_{i=1}^{c-1}N_i})^{-1}}$ taking care of the enumeration
over the $\mcC_{\geq0}$. Thus, by \eqref{equ:zeta.nonneg},
\begin{align*}
  \zeta_{\mcE(\lri)\car \mcL(\lri)}(s) &= \frac{1}{1-q^{-ns}}\cdot
  \frac{1}{1-q^{-s\sum_{i=1}^{c-1}N_i}}\cdot
  \frac{(1-q^{-s\sum_{i=1}^{c-1}N_i})(1-q^{-s\sum_{i=1}^{c-1}(n-N_i)})}{1-q^{-s(c-1)n}}A^{\nl}(s)\\ &=
  \frac{1-q^{-s\sum_{i=1}^{c-1}(n-N_i)}}{(1-q^{-ns})(1-q^{-s(c-1)n})}A^{\nl}(s)
  = \frac{1}{1-q^{-ns}} \Xi(s) A^{\nl}(s).
\end{align*}
Given \eqref{equ:funeqA}, the functional
equation~\eqref{equ:funeq.submodule} follows as, trivially,
\begin{equation}\label{equ:funeqtrivial}
  \left. \frac{1-q^{-s\sum_{i=1}^{c-1}(n-N_i)}}{(1-q^{-ns})(1-q^{-s(c-1)n})}
  \right|_{q \rarr q^{-1}} = -q^{-s\sum_{i=0}^{c-1}N_i}
  \frac{1-q^{-s\sum_{i=1}^{c-1}(n-N_i)}}{(1-q^{-ns})(1-q^{-s(c-1)n})}.
\end{equation}

It remains to devise the weight function $w$ to the requirement of
\eqref{equ:weight}.
\begin{definition}
  Given $M\in\GL_n(K_{\mfp})\cap \Mat_n(\lri)$ corresponding to a
  maximal lattice $\Lambda$, write $M=\left( M^{(ij)}
  \right)_{i,j\in[c]}$ as in the proof of
  Lemma~\ref{lem:delta}. Define
$$m_2([\Lambda]) = \min\{
v(M^{(ic)})\mid i\in [c]\}.$$
\end{definition}

Informally speaking, $m_2([\Lambda])$ is the valuation of the last
$n_c$ columns of a matrix $M$ representing $\Lambda_{\max}$. Recall
the ``distance'' function $\wt{m}_1:\mcV_n \rarr \N_0$; see
Definition~\ref{def:dist}.

\begin{lemma}\label{lem:m}
  Equation~\eqref{equ:weight} holds with $w([\Lambda]) = (c-1)
  \wtm_1([\Lambda]) - m_2([\Lambda])$.
\end{lemma}

\bp Denote by $[\Lambda_0]$ the generator of $\mcC_{\geq 0}$ and by
$\Lambda_{0,\max}\in[\Lambda_0]$ its unique maximal element. It
suffices to observe that $v( M \delta^{\wtm_1([\Lambda])}) =
m_2([\Lambda])$. Hence the matrix $M \delta^{\wt{m}_1([\Lambda])}$
corresponds to the lattice $\pi^{{m}_2([\Lambda])}\Lambda_{0,\max}$,
whence $$\pi^{-m_2([\Lambda])} \left( M
\delta^{\wt{m}_1([\Lambda])} \right) \wtdelta^{\wt{m}_1([\Lambda])} = M
\pi^{(c-1)\wt{m}_1([\Lambda]) - \wt{m}_2([\Lambda])}$$ corresponds to
$\Lambda_{0,\max}\wtdelta^{\wt{m}_1([\Lambda])}$.\ep

For later reference we record another formula for the invariant
$m_2$. Write $M=D \alpha^{-1}$ as above.  Setting, for $r\in[n]$,
$$v^{(2)}_r(\alpha) := \min\left\{ v\left((\alpha^{\textup{adj}})_{\rho\sigma}\right) \mid \rho \geq r, \sigma\in\, ]N_1,n]\right\},$$
we obtain
\begin{equation*}%\label{equ:m2}
  m_2([\Lambda]) = \min\left\{\sum_{\iota\in I}r_{\iota}, \sum_{r \leq \iota\in I}r_{\iota} +
    v^{(2)}_r(\alpha) \mid r\in[n]\right\}.
\end{equation*}

To compute $A^{\nl}(s)$ we need, given a lattice class
$[\Lambda]\in\mcV_n$ with $\nu([\Lambda])=(I,\bfr_I)$, to keep track
of the quantity
\begin{multline*}
  q^{-s\left( (\sum_{\iota\in I}\iota r_{\iota}) +
    nw([\Lambda])\right)} = q^{-s\left( (\sum_{\iota\in I}\iota
    r_\iota) + n((c-1)\wtm_1([\Lambda]) - m_2([\Lambda]))\right)}
  =\\ q^{-s\left( (\sum_{\iota\in I} r_{\iota}(\iota +
    n(c-1)))-n((c-1)m_1([\Lambda]) + m_2([\Lambda]))\right)}.
\end{multline*}
(Here we used \eqref{equ:index}, Lemma~\ref{lem:m}, and
\eqref{equ:m1.tilde}.)  To this end we define, given $I\subseteq
[n-1]_0$ and $\bfr\in\N^{|I|}$ as above, for
$\bfm=(m_1,m_2)\in\N_0^2$,
\begin{equation*}%\label{equ:Aid}
 \mcN^{\nl}_{I,\bfr,\bfm} = \left| \left\{[\Lambda] \in \mcV_n \mid
 \nu([\Lambda]) = (I,\bfr),\, m_i([\Lambda]) = m_i, \,
 i\in\{1,2\}\right\}\right|
\end{equation*}
and set
\begin{equation}\label{def:AidI}
  A_{I}^{\nl}(s) = \sum_{\bfr\in\N^{|I|}} q^{-s\sum_{\iota\in
      I}r_\iota(\iota+n(c-1))}
  \sum_{\bfm=(m_1,m_2)\in\N_0^2}\mcN^{\nl}_{I,\bfr,\bfm}q^{sn\left((c-1)m_1+m_2\right)},
\end{equation} 
so that $A^{\nl}(s) = \sum_{I\subseteq [n-1]}A_I^{\nl}(s)$.

\subsubsection{$\mfp$-Adic integration}
To establish the functional equation \eqref{equ:funeqA} for
$A^{\nl}(s)$ we express each of the functions $A_I^{\nl}(s)$ in terms
of $\mfp$-adic integrals of the form \cite[(6)]{Voll/10} that satisfy
the hypotheses of \cite[Theorem~2.3]{Voll/10}. \revtwo{To this end we
  consider multivariate $\mfp$-adic integrals $Z^{\nl}_{I}(\bfs)$
  (cf.\ \eqref{equ:Z}) on $\mfp^{|I|} \times \Gamma$ whose integrands
  are defined using carefully chosen polynomial mappings. The latter
  are designed in such a way that the numbers
  $\mcN^{\nl}_{I,\bfr,\bfm}$ entering \eqref{def:AidI} may easily be
  expressed in terms of the (Haar) measures $\mu^{\nl}_{I,\bfr,\bfm}$
  of the subsets of $\mfp^{|I|} \times \Gamma$ on which the integrands
  are constant; cf.\ Lemma~\ref{lem:mass=zahl}. The generating
  functions $A_I^{\nl}(s)$ may be recovered from (suitable
  normalizations of) the integrals $Z^{\nl}_I(\bfs)$ upon a
  judiciously chosen affine-linear univariate substitution of
  variables; cf.\ \eqref{equ:A=Z}.  

  This type of connection between generating functions and $\mfp$-adic
  integrals has been exploited numerous times before. It links, for
  instance, Igusa's local zeta function (cf.\
  Section~\ref{subsec:ilzf}) with the problem of counting solutions to
  polynomial congruences modulo powers of $\mfp$; cf.\
  \cite[Section~1.2]{Denef/91}. The formalism set up in
  \cite[Section~2.2]{Voll/10} is based on the very same principle.  }

To return to the problem at hand we define, for $i,r\in[n]$, \revtwo{sets of polynomials}
\begin{align*}
\bff^{(1)}_{i,r}(\bfy) &= \left\{ \left( \mcR_{(\iota)}(\bfy)\right)_{\rho\sigma} \mid \iota \leq i, \rho \geq r, \sigma \in [d]\right\},\\
\bff^{(2)}_{r}(\bfy) &= \left\{ \left(\bfy^{\textup{adj}})\right)_{\rho\sigma} \mid \rho\geq r, \sigma\in\,]N_1,n]\right\},
\end{align*} 
and set, for $I\subseteq[n-1]$,
\begin{alignat*}{2}
  \bfg^{(1)}_{n,I}(\bfx,\bfy) &= \left\{\prod_{\iota\in I}x_{\iota}\right\} \cup &\bigcup_{(i,r)\in[n]^2}\left( \prod_{\iota\in I}x_\iota^{\delta_{r \leq \iota} + \delta_{i > \iota}}\right) &\bff^{(1)}_{i,r}(\bfy),\\
  \bfg^{(2)}_{n,I}(\bfx,\bfy) &= \left\{\prod_{\iota\in
      I}x_{\iota}\right\} \cup & \bigcup_{r\in[n]}\left(
    \prod_{\iota\in I}x_\iota^{\delta_{r \leq \iota} }\right) & \bff^{(2)}_{r}(\bfy),
\end{alignat*}
and, for $\kappa\in[n-1]$,
$$\bfg_{\kappa,I}(\bfx,\bfy) = \left\{\prod_{\iota\in I}x_{\iota}^{\delta_{\iota\kappa}}\right\}.$$
With this data we define the $\mfp$-adic integral
\begin{multline}\label{equ:Z}
Z^{\nl}_{I}(\bfs) = Z^{\nl}_{I}((s_\iota)_{\iota\in
  I},s^{(1)}_n,s^{(2)}_n) :=\\\int_{\mfp^{|I|} \times \Gamma} \|
\bfg^{(1)}_{n,I}(\bfx,\bfy) \|^{s^{(1)}_n} \|
\bfg^{(2)}_{n,I}(\bfx,\bfy) \|^{s^{(2)}_n}
\prod_{\kappa\in[n-1]}\|g_{\kappa,I}(\bfx,\bfy)\|^{s_\kappa} |\tud
\bfx_I| | \tud \bfy|.
\end{multline}
 Here, $\bfs = (s_1,\dots,s_{n-1},s^{(1)}_n,s^{(2)}_n)$ is a vector of
 complex variables; note, however, that $s_{\kappa}$ occurs on the
 right-hand side if and only if $\kappa\in I$. Moreover, $|\tud
 \bfx_I| | \tud \bfy|$ denotes the Haar measure, normalized such that
 the domain of integration has measure~$q^{-|I|}\mu(\Gamma)$, where
 $\mu(\Gamma) = \prod_{i=1}^n(1-q^{-i})$. \revtwo{By $\| \cdot \|$ we
   denote the $\mfp$-adic (maximum) norm.} The integral
 $Z^{\nl}_{I}(\bfs)$ is, by design, of the form
 \cite[(6)]{Voll/10}. Discarding at most finitely many primes, we may
 assume that the assumptions of \cite[Theorem~2.2]{Voll/10} are
 satisfied. This will imply that the normalized integrals
\begin{equation*}%\label{equ:Z.normalized}
 \wt{Z^{\nl}_{I}}(\bfs) :=
\frac{Z^{\nl}_{I}(\bfs)}{(1-q^{-1})^{|I|}\mu(\Gamma)}
\end{equation*} (cf.\
\cite[(10)]{Voll/10}) satisfy the `inversion properties' established
in \cite[Theorem~2.3]{Voll/10}.  Hence the sum $\wt{Z^{\nl}}(\bfs) :=
\sum_{I\subseteq[n-1]}\binom{n}{I}_{q^{-1}} \wt{Z^{\nl}_I}(\bfs)$
(cf.\ \cite[(16)]{Voll/10}) satisfies the functional equation
\begin{equation}\label{equ:funeq.ZItilde}
\left.\wt{Z^{\nl}}(\bfs)\right|_{q\rarr q^{-1}} = (-1)^{n-1}q^{\binom{n}{2}}\wt{Z^{\nl}}(\bfs);
\end{equation}
cf.\ \cite[Cor.~2.3]{Voll/10}. Here, $\binom{n}{I}_X\in\Z[X]$ denotes
the Gaussian multinomial coefficient.

It remains to show that, for each $I\subseteq [n-1]$, the generating
function $A_{I}^{\nl}(s)$ is indeed obtainable from the $\mfp$-adic
integral $Z^{\nl}_{I}(\bfs)$ by a suitable specialization of the
variables~$\bfs$. We start by measuring the sets on which the
integrand of $Z^{\nl}_I(\bfs)$ is constant. More precisely we set, for
$\bfm=(m_1,m_2)\in\N_0^2$ and $\bfr\in\N^{|I|}$,
$$\mu^{\nl}_{I,\bfr,\bfm} := \mu\left\{(\bfx,\bfy) \in \mfp^{|I|} \times \Gamma \mid \forall \iota\in I: v(x_\iota) = r_{\iota}, \bfm(\bfx,\bfy) = (m_1,m_2)\right\},$$
where $\bfm(\bfx,\bfy) = \left(
  \bfm(\bfx,\bfy)_1,\bfm(\bfx,\bfy)_2\right)$ and
\begin{align*}
  \bfm(\bfx,\bfy)_1 &= \min \left\{ \sum_{\iota\in I}r_{\iota},
  \sum_{r\leq \iota\in I}v(x_\iota) + \sum_{i > \iota\in I}v(x_\iota)
  + v^{(1)}_{i,r}(\bfy) \mid
  (i,r)\in[n]^2\right\},\\ \bfm(\bfx,\bfy)_2 &= \min \left\{
  \sum_{\iota\in I}r_{\iota}, \sum_{r\leq \iota\in I}v(x_\iota)
  \phantom{ + \sum_{i > \iota\in I}v(x_\iota), }+ v^{(2)}_{r}(\bfy)
  \mid r\in[n]\right\}.
\end{align*}

Then, by design,
\begin{equation}\label{equ:Zsum}
 \wt{Z^\nl_I}(\bfs) =
 \frac{1}{(1-q^{-1})^{|I|}\mu(\Gamma)}\sum_{\bfr\in\N^{|I|}}
 q^{-\sum_{\iota\in I}s_\iota r_\iota}\sum_{\bfm = (m_1,m_2)
   \in\N_0^2}\mu^{\nl}_{I,\bfr,\bfm}q^{-s^{(1)}_nm_1 - s^{(2)}_nm_2}.
\end{equation} 
The numbers $\mu^{\nl}_{I,\bfr,\bfm}$ are closely related to the
natural numbers $\mcN^{\nl}_{I,\bfr,\bfm}$ we are looking to
control. 
\begin{lemma}\label{lem:mass=zahl}
\begin{equation}\label{equ:Mass=Zahl}
  \mcN^{\nl}_{I,\bfr,\bfm} =
  \frac{\binom{n}{I}_{q^{-1}}}{(1-q^{-1})^{|I|}\mu(\Gamma)}
  \mu^{\nl}_{I,\bfr,\bfm} q^{\sum_{\iota\in I}r_{\iota}(\iota(n-\iota)
    + 1)}.
\end{equation}
\end{lemma}

\begin{proof}
Analogous to \cite[Lemma~3.1]{Voll/10}.
\end{proof}
Thus, combining \eqref{def:AidI}, \eqref{equ:Mass=Zahl}, and
\eqref{equ:Zsum}, we obtain
\begin{align}
  A^{\nl}_I(s) &= \sum_{\bfr\in\N^{|I|}} q^{-s\sum_{\iota\in
      I}r_\iota(\iota+n(c-1))} \sum_{\bfm = (m_1,m_2)
    \in\N_0^2}\mcN^{\nl}_{I,\bfr,\bfm}q^{sn\left((c-1)m_1+m_2\right)}\nonumber\\ &=
  \frac{\binom{n}{I}_{q^{-1}}}{(1-q^{-1})^{|I|}\mu(\Gamma)}
  \sum_{\bfr\in\N^{|I|}} q^{-\sum_{\iota\in I}r_\iota(s(\iota+n(c-1)) -
    \iota(n-\iota)-1)}
  \sum_{\bfm\in\N_0^2}\mu^{\nl}_{I,\bfr,\bfm}q^{sn\left((c-1)m_1+m_2\right)}\nonumber\\ &=
  \binom{n}{I}_{q^{-1}} \wt{Z^{\nl}_I} \left(
  \left(s(\iota+n(c-1))-\iota(n-\iota)-1\right)_{\iota\in I},-ns(c-1)
  , -ns \right).\label{equ:A=Z}
\end{align}
The functional equation \eqref{equ:funeqA} for $A^{\nl}(s) =
\sum_{I\subseteq [n-1]}A^{\nl}_I(s)$ now follows from
\eqref{equ:funeq.ZItilde}. This concludes the proof of
Theorem~\ref{thm:main}.

\subsection{Proof of Corollary~\ref{cor}}\label{subsec:cor}

The proof of Theorem~\ref{thm:main} expresses the relevant local
submodule zeta functions in terms of $\mfp$-adic integrals which are
known to be representable by formulae of \emph{Denef-type}. In other
words, there exist algebraic varieties $V_i$, defined over~$\Gri$, and
rational functions $W_i(X,Y)\in\Q(X,Y)$ for $i=1,\dots,M$, such that
the following holds. For almost all $\mfp$ and all $\Lri$ we have
$$\zeta_{\mcE(\Lri)\car \mcL(\Lri)}(s) = \sum_{i=1}^M
|\overline{V_i}(\Lri/\mfP)| W_i(q^f,q^{-fs})=
\frac{P_{\Lri}(q^{-fs})}{Q_{\Lri}(q^{-fs})}$$ for coprime polynomials
$P_\Lri(Z), Q_\Lri(Z)\in \Z[Z]$, say
$$P_{\Lri}(Z) = \sum_{i=0}^{\deg_Z P_{\Lri}}\alpha_iZ^i, \quad
Q_{\Lri}(Z)=\prod_{j=1}^r(1-q^{f a_j}Z^{b_j}) = \sum_{i=0}^{\deg_Z
  Q_{\Lri}}\beta_i Z^i$$ for integers $a_j\in\N_0$, $b_j\in\N$,
$j\in[r]$. Note that the coefficients $\beta_i$ of $Q_{\Lri}$ are
uniformly given by polynomials in~$q^f$ with integral coefficients
which are independent of~$\Lri$. (The pair $(\mcL,\mcE)$ is almost
uniform if and only the same holds for the coefficients $\alpha_i$ of
$P_{\Lri}$.) Moreover, the degree $\delta_2 := \deg_Z Q_{\Lri}$ is
independent of $\Lri$.  Note that $\alpha_0=\beta_0=1$, as the
generating series $\zeta_{\mcE(\Lri)\car \mcL(\Lri)}(s)$ has evidently
constant term~$1$. The functional equation \eqref{equ:funeq.submodule}
implies that $\delta_1:=\deg_Z P_{\Lri}$, too, is independent of
$\Lri$ and there exist $c_1,c_2\in\N_0$, $\sigma_1,
\sigma_2\in\{0,1\}$ satisfying $c_2-c_1=\binom{n}{2}$ and $\sigma_1 +
\sigma_2 \equiv n \bmod(2)$ such that
\begin{alignat*}{2}
  \left. \alpha_i \right|_{q\rarr q^{-1}} &= (-1)^{\sigma_1}
  q^{-fc_1}\alpha_{\delta_1-i},& \quad \textrm{ for }i&\in [\lfloor
    \delta_1/ 2 \rfloor], \\ \left. \beta_i \right|_{q\rarr q^{-1}} &=
  (-1)^{\sigma_2} q^{-fc_2}\beta_{\delta_2-i},& \quad \textrm{ for
  }i&\in [\lfloor \delta_2/ 2 \rfloor], \end{alignat*} and $ \delta_1
- \delta_2 = \deg_{q^{-fs}}\left( \zeta_{\mcE(\Lri) \car
  \mcL(\Lri)}(s)\right) = - \sum_{i=0}^{c-1}N_i$. Corollary~\ref{cor}
follows.

\section{Necessary vs.\ sufficient conditions for local functional equations}\label{sec:nec.vs.suff}
Theorem~\ref{thm:main} gives sufficient criteria for generic local
functional equations for submodule zeta functions associated to
nilpotent algebras of endomorphisms. Moreover, all examples of such
submodule zeta functions for which we know explicit formulae for the
generic Euler factors are consistent with the speculation that the
hypotheses of Theorem~\ref{thm:main} are also necessary for local
functional equations of the form \eqref{equ:funeq.submodule} to
hold. An analogy with Igusa's local zeta function, however, sketched
in Section~\ref{subsec:ilzf}, suggests caution. An insufficient but
potentially necessary criterion for local functional equations for
zeta functions in terms of so-called reduced zeta functions is
discussed in Section~\ref{subsec:red}.

\subsection{Igusa's local zeta function} \label{subsec:ilzf} The proof
of Theorem~\ref{thm:main} is ultimately inspired by Denef's and
Meuser's proof of a functional equation for Igusa's local zeta
functions associated to a homogeneous polynomial. Let
$F\in\Gri[X_1,\dots,X_n]$, let $\mfp$ be a nonzero prime ideal of
$\Gri$ of index $q$ in $\Gri$ and $\Lri$ be a finite extension of
$\Gri_{\mfp}$, with maximal ideal $\mfP$ of index $q^f$ in
$\Lri$. Igusa's local zeta function associated to $F$ at $\mfP$
is $$Z_{F,\mfP}(s) = \int_{\Lri^n}|F(\bfx)|_{\mfP}^s |\tud \bfx|,$$
where $|\;|_{\mfP}$ is the $\mfP$-adic norm and $|\tud \bfx|$ denotes
the additive Haar measure on $\Lri^n$, normalized such that $\Lri^n$
has measure~$1$. It is a rational function in~$q^{-fs}$. \revone{If}
$F$ is homogeneous of degree~$d$, then, for almost all~$\mfp$ and all
$\Lri$, the functional equation
\begin{equation}\label{equ:funeq.igusa}
  \left. Z_{F,\mfP}(s)\right|_{q \rarr q^{-1}} = q^{-fds}Z_{F,\mfP}(s)
\end{equation}
holds; cf.~\cite[Theorem~4]{DenefMeuser/91}. (Note that in
\cite{DenefMeuser/91} the inertia degree $f=f(\Lri,\Gri_\mfp)$ is
denoted by~$e$.) Here, as in Theorem~\ref{thm:main}, equation
\eqref{equ:funeq.igusa} is explained in terms of a suitable formula of
the form \eqref{equ:denef} and the Weil conjectures for smooth
\emph{projective} algebraic varieties over finite fields; cf.\
Remark~\ref{rem:invert}. Projectivity of the relevant varieties is a
consequence of the homogeneity of $F$. This line of argument breaks
down if $F$ is not homogeneous (\revone{and does not become
  homogeneous} after an affine transformation of polynomial
variables), and simple examples illustrate that functional equations
such as \eqref{equ:funeq.igusa} will not hold in general. We do not
know of general necessary conditions.

Note that the functional equation \eqref{equ:funeq.igusa} implies an
analogue of \eqref{equ:deg.t}, viz.\ that
\begin{equation}\label{equ:deg.igusa}
  \deg_{q^{-fs}}(Z_{F,\mfP}(s))=-d,
\end{equation}
a fact first proven in \cite{Denef/87}. We are not aware of general
results on the degree of Igusa's local zeta function for
nonhomogeneous polynomials, but again simple examples show that
\eqref{equ:deg.igusa} is not universal. This contrasts with
Conjecture~\ref{con:deg}, which predicts the degree of generic local
submodule zeta functions associated with nilpotent endomorphism
algebras, regardless of whether or not they satisfy functional
equations.

\subsection{Reduced zeta functions}\label{subsec:red}
In \cite{Evseev/09}, Evseev introduced ``reduced zeta functions''
associated to various enumeration problems pertaining to
finite-dimensional Lie algebras. His constructions apply quite
generally to zeta functions $Z(s) = \prod_{\mfp}Z_\mfp(s)$ which
satisfy Euler products, indexed by the nonzero prime ideals $\mfp$ of
a number ring, \revone{with Euler factors which are rational
  generating} functions in $q^{-s}$ whose coefficients can be
interpreted in terms of suitably ``geometric functions'' in the
respective residue fields, such as formulae of the
form~\eqref{equ:denef}. Informally speaking, the reduced zeta function
$Z_{\red}(T)$ is a rational function in $T$ obtained by ``setting
$q=1$'' in such formulae, whilst treating the parameter $q^{-s}$ as an
independent formal variable~$T$. If $Z(s)$ is ``almost uniform'' (cf.\
Definition~\ref{def:uni}), i.e.\ if there exists $W\in\Q(X,Y)$ such
that $Z_{\mfp}(s) = W(q,q^{-s})$ for almost all $\mfp$, then
$Z_{\red}(T)=W(1,T)$. The formal definition given in \cite{Evseev/09}
uses Euler-Poincar\'e characteristics in a motivic setting. We remark
that the concept of reduced zeta function seems subtly related to the
concept of topological zeta function; cf.~\cite{Rossmann/15}.

Under certain conditions, reduced zeta functions satisfy functional
equations upon inversion of $T$ which reflect geometric properties of
polyhedral cones. Indeed, in \cite{Evseev/09}, Evseev describes
sufficient conditions for a reduced zeta function to be the Hilbert
series $H(T)$ of a graded ring $R_{\mcC}$ associated to a rational
polyhedral cone $\mcC\subseteq \R_{\geq 0}^n$;
cf.\ \cite[Proposition~4.1]{Evseev/09}.  Up to a sign, the rational
function resulting from the inversion of $T$ is the Hilbert series of
the graded ring $R_{\mcC^{\circ}}$ associated to the \emph{interior}
of $\mcC$. A functional equation of the form
$$H(1/T) = \pm T^\beta H(T)$$ occurs if and only \revone{if} $\mcC$ has a unique
minimal integral interior vector. For further details and an
interpretation of these facts in terms of commutative algebra, viz.\
the language of Cohen-Macaulay and Gorenstein rings and modules, see
\cite[Chapter~1]{Stanley/96}. For an application in the context of
ideal zeta functions of nilpotent Lie lattices, see
Section~\ref{subsec:max.class}.

A functional equation for the reduced zeta function is, in general,
not sufficient for generic $\mfp$-adic functional equations; cf.\
Example~\ref{exa:M4.vs.Fil4}. If, however, $Z(s)$ is almost uniform as
defined above, then a functional equation for the reduced zeta
function $Z_{\red}(T)$ is a necessary condition for functional
equations for the $\mfp$-adic zeta functions
$Z_{\mfp}(s) = W(q,q^{-s})$. We speculate that this also holds without
the hypothesis of almost uniformity.

In the light of the analogy with the theory of Hilbert series
associated to rational polyhedral cones sketched above, it is tempting
to interpret the result of the $\mfp$-adic zeta function obtained by
``inverting $q$'' in terms of a generating function enumerating
``interior points'', too. In this spirit, in
\cite[Definition~4.26]{duSWoodward/08} du Sautoy and Woodward
introduced the notion of $p$-\emph{ideal} of a $\Zp$-Lie lattice, in
analogy to the notion of interior point of a polyhedral cone. Their
hope clearly was to interpret local functional equations in terms of
natural correspondences between the ideal- and $p$-ideal-lattices. Our
proof of Theorem~\ref{thm:main}, however, relies in an essential
manner on geometric properties of smooth projective algebraic
varieties, as established by the Weil conjectures (see
Remark~\ref{rem:invert}). It seems therefore that any such
interpretation of the local functional equations established in
Theorem~\ref{thm:main} would have to interpolate between these deep
algebro-geometric symmetries on the one hand and the symmetries
satisfied by Hilbert series associated to cones on the other.

\section{Ideal zeta functions of nilpotent Lie lattices}\label{sec:nil}

In this section we discuss applications of Theorem~\ref{thm:main} to
ideal zeta functions of nilpotent Lie lattices.  Ideal zeta functions
of nilpotent Lie rings have been introduced in \cite{GSS/88} as tools
in the study of normal subgroup growth of finitely generated nilpotent
groups. The technical tool facilitating this linearization is the
Mal'cev correspondence \revone{mentioned in
  Section~\ref{subsec:appI}}. Hence, all results in this section on
ideal zeta functions of nilpotent Lie lattices have immediate
consequences on normal subgroup zeta functions of finitely generated
nilpotent groups. Only in Corollary~\ref{cor:free} did we choose to
spell out such a consequence (of Theorem~\ref{thm:free}, in this
instance).

Assume that $\mcL$ is a nilpotent $\Gri$-Lie lattice satisfying
Assumption~\ref{ass} -- which, as we saw in Remark~\ref{rem:ass}, is
vacuous for Lie rings \revone{(i.e.\ if $\Gri=\Z$)} --, with Lie
bracket~$[\,,]$. Recall that $(Z_i)_{i=0}^c$ is the upper central
series of~$\mcL$. In particular, $c$ is the nilpotency class of $\mcL$
and $N_i = \rk_{\Gri}(\mcL/Z_i)$ for all $i\in[c]_0$. As noted in
Remark~\ref{rem:con.sat}, Condition~\ref{con} is trivially satisfied
if~$c\leq 2$. In these cases, Theorem~\ref{thm:main} confirms known
results. For $c=1$, see Example~\ref{exa:c=1}. For $c=2$,
Theorem~\ref{thm:main} in this setting is (a mild generalization of)
\cite[Theorem~C]{Voll/10}.

\subsection{Examples without functional
  equations}\label{subsec:nil.exa} That local functional equations
akin to \eqref{equ:funeq.submodule} may fail in nilpotency class
greater than $2$ was first discovered by Woodward.

\begin{example}\label{exa:Fil4}
  Consider the class-$4$-nilpotent Lie ring
$$\mcL = \textup{Fil}_4 = \la z,x_1,x_2,x_3,x_4 \mid [z,x_1] = x_2,
[z,x_2]=x_3,[z,x_3]=x_4,[x_1,x_2]=x_4\ra_{\Z}.$$
Here, as well as in comparable Lie lattice presentations throughout
the paper, \revone{products among generators} other than those
following -- by antisymmetry or the Jacobi identity -- from the given
ones are assumed to be trivial. In the given example this means that
$x_4$ is central and $x_3$ commutes with both $x_1$
and~$x_2$. \cite[Theorem~2.39]{duSWoodward/08} gives explicit formulae
for the local ideal zeta functions
$\zeta^{\nl}_{\textup{Fil}_4(\Zp)}(s)$, valid for all primes~$p$. They
are all given by a single rational function $W(X,Y)\in\Q(X,Y)$ upon
the substitution $X=p$, $Y=p^{-s}$ and do not satisfy a functional
equation of the form~\eqref{equ:funeq.submodule}. The associative
algebra $\mcE$ generated by $\ad(\Fil_{4})$ does not satisfy the
conditions of Theorem~\ref{thm:main}. Indeed, it is generated by
$\ad(z)$ and $\ad(x_1)$ which, with respect to the chosen cocentral
$\Z$-basis $(z,x_1,x_2,x_3,x_4)$ of $\Fil_4$ are represented by the
integral $5\times 5$-matrices
$$C_1 = \left( \begin{array}{cc|c|c|c}
  \phantom{0}&\phantom{0}&&&\\&&-1&&\\\hline &&&-1& \\ \hline
  &&&&-1\\\hline&&&&\end{array}\right)\quad \textup{ and } \quad C_2 =
\left( \begin{array}{cc|c|c|c}
  \phantom{0}&\phantom{0}&1&&\\&&&&\\\hline &&&\phantom{0}&1\\\hline
  &&&&\\\hline&&&&\end{array}\right),$$ respectively. Here, the block
structure reflects the decomposition $\mcL = \bigoplus_{i=1}^4 \mcL_i$
with $\mcL_1 = \la z,x_1\ra_{\Z}$ and $\mcL_i = \la x_i\ra_{\Z}$ for
$i\in\{2,3,4\}$. \revtwo{The matrix $C_2$ is not as prescribed in
  Condition~\ref{con2}, as it is} not supported solely on the first
block-off diagonal: $(C_2)_{34}\neq 0$. The failure of the functional
equation reflects the (easily verifiable) fact that no other
decomposition or choice of generators for $\mcE$ mitigates this
failure.
\end{example}

\begin{example}
Consider the class-$3$-nilpotent Lie ring
$$\mcL= \mfg_{6,6} = \la x_1,\dots,x_6 \mid [x_1,x_2]=x_4,
[x_1,x_3]=x_5, [x_1,x_4]=x_6, [x_2,x_3]=x_6\ra_{\Z}.$$
As observed in \cite[Example~4.58]{duSWoodward/08},
$Z_1 = \la x_5,x_6\ra_{\Z}$, $Z_2 = \la x_3,x_4,x_5,x_6\ra_{\Z}$,
$Z_3=\mcL$. Set $\mcL_1 = \la x_1,x_2\ra_{\Z}$,
$\mcL_2 = \la x_3,x_4\ra_{\Z}$, $\mcL_3 = \la x_5,x_6\ra_{\Z}$. Whilst
$c_1=\ad(x_1)$ and $c_2 = \ad(x_2)$ satisfy condition
\eqref{equ:shift}, $c_3 = \ad(x_3)$ does not. Moreover, this failure
is independent of the specific choice of complements
$\mcL_i$. Crucially, $c_3$ is not contained in the associative algebra
generated by $c_1$ and $c_2$. This is consistent with the fact that
the ideal zeta function of $\mfg_{6,6}$ does not satisfy the
conclusions of Theorem~\ref{thm:main};
see~\cite[Theorem~2.44]{duSWoodward/08}.
\end{example}

\begin{remark}\label{rem:local.base.ext}
  Like many others of their kind, the computations in
  \cite{duSWoodward/08} are only carried out for local rings
  $\lri=\Zp$. The resulting formulae, however, also cover the case of
  general finite extensions $\lri$ of $\Zp$; one replaces $p$ by the
  residue field cardinality~$q$. See \cite[Section~2.3]{Rossmann/15a}
  for a formal justification of this stability under local base
  extension.
\end{remark}

Numerous further formulae for local ideal zeta functions of nilpotent
Lie rings lacking a generic local functional equation can be found in
\cite[Section~2]{duSWoodward/08}. Like most formulae recorded
in~\cite{duSWoodward/08}, they are obtained by computations with
$p$-adic {cone integrals}; cf.\ \cite{duSG/00}.  In
\cite[Conjecture~4.5]{duSWoodward/08} du Sautoy and Woodward formulate
a sort of conjecture on such cone integrals that would imply
functional equations akin to~\eqref{equ:funeq.submodule}. The
conjecture's hypotheses, however, suffer from a degree of
illdefinedness (``Suppose that [two specified conditions] and some
as-yet-undetermined conditions hold.''). It takes indeed a ``cavalier
attitude to the incompleteness of Conjecture~4.5''
(\cite[p.~98]{duSWoodward/08}) to use it to speculate about the
occurrence of local functional equations of ideal zeta functions of
Lie rings.

The shortcomings of \cite[Conjecture~4.5]{duSWoodward/08}
notwithstanding, the present paper owes a great deal of inspiration to
\cite[Chapter~4]{duSWoodward/08}. Our very Condition~\ref{con} is
modelled on the conjunction of the properties $(\dag)$ and $(*)$ in
\cite[Definition~4.56]{duSWoodward/08}; our matrix $\delta$ (see
\eqref{def:delta}) represents -- up to a scalar factor -- the map
$\pi_{\mathcal{B}}$ in \cite[Definition~4.40]{duSWoodward/08} with
respect to a cocentral basis, a concept generalizing
\cite[Definition~4.37]{duSWoodward/08}. There are, nonetheless, two
fundamental differences in approach. Firstly, we do not analyze cone
integrals but develop the $\mfp$-adic machinery introduced
in~\cite{Voll/10}. Secondly, we realize that the ideal zeta function
of a nilpotent Lie lattice $\mcL$ is determined by the associative
algebra generated by $\ad(\mcL)$; rather than hypothesizing about
linear bases for the former, we formulate necessary conditions on
suitable generators of the latter.

In the remainder of the section we develop a number of (unconditional)
applications of Theorem~\ref{thm:main} establishing generic local
functional equations for ideal zeta functions of nilpotent Lie
lattices, confirming some of the more specific conjectures in
\cite{duSWoodward/08}. \revone{We record these results in the
  characteristic-independent fashion made possible by
  Corollary~\ref{cor:char.p}. Recall that, given a Lie ring
  (viz.\ $\Z$-Lie lattice) $\mcL$ and a compact discrete valuation
  ring $\lri$ (of arbitrary characteristic), the ideal zeta function
  $\zeta^{\nl}_{\mcL(\lri)}(s)$ of $\mcL(\lri) = \mcL\otimes_\Z\lri$
  is the rational ordinary generating function enumerating the
  $\lri$-ideals of $\mcL(\lri)$ of finite additive index
  in~$\mcL(\lri)$.}

\subsection{Free nilpotent Lie rings}\label{subsec:free.Lie.ring}
Given $c,d\in\N$, consider the free class-$c$-nilpotent Lie ring
$\mff_{c,d}$ on free Lie generators~$x_1,\dots,x_d$. The well-known
formula for the $\Z$-ranks of the sections of the terms
$\gamma_j(\mff_{c,d})$ of the lower central series of $\mff_{c,d}$
(which coincides with the upper central series of $\mff_{c,d}$), due
to Witt (\cite[Satz~3]{Witt/37}), implies that, for $i\in[c]_0$,
$$\rk_{\Z}\left( \mff_{c,d} / \gamma_{c-i+1}(\mff_{c,d})\right) =
\sum_{1 \leq j \leq c-i} \frac{1}{j}\sum_{k|j}\mu(k)d^{j/k} = N_i,$$
the numbers defined in~\eqref{equ:mobius}.  The following consequence
of Theorem~\ref{thm:main} \revone{and its Corollary~\ref{cor:char.p}}
implies the conclusion of \cite[Theorem~1.3]{duSWoodward/08} without
relying on the incomplete \cite[Conjecture~4.5]{duSWoodward/08}.

\begin{theorem}\label{thm:free} 
  For almost all primes $p$ \revone{and all compact discrete valuation
    rings $\lri$ of residue characteristic~$p$,}
$$\left.\zeta^{\nl}_{\mff_{c,d}(\lri)}(s) \right|_{p\rarr p^{-1}} =
(-1)^{N_0}q^{\binom{N_0}{2}-s\left(\sum_{i=0}^{c-1}N_i\right)}\zeta^{\nl}_{\mff_{c,d}(\lri)}(s).$$
\end{theorem}

\begin{proof} 
  We choose a Hall ($\Z$-)basis $\mcH$ on $\{x_1,\dots,x_d\}$ for
  $\mff_{c,d}$; cf.\ \cite{Hall/50}. By construction, elements of
  $\mcH$ are Lie monomials in $x_1,\dots,x_d$ with a well-defined
  degree, or \emph{weight} $i\in[c]$. For $i\in[c]$, write $\mcH^i$
  for the set of elements of $\mcH$ of weight~$i$. For instance,
  $\mcH^1 = \{x_1,\dots,x_d\}$. Note that
  $\bigcup_{\iota\leq i}\mcH^\iota$ has cardinality $N_i$ for each
  $i\in[c]_0$. Setting $\mcL_{i}:=\la \mcH^i\ra_{\Z}$ for $i\in[c]$
  and $\{c_1,\dots,c_d\} = \{\ad(x_1),\dots,\ad(x_d)\}$,
  Condition~\ref{con} is clearly satisfied. The result thus follows
  from Theorem~\ref{thm:main} \revone{and Corollary~\ref{cor:char.p}}.
\end{proof}

Explicit formulae for $\zeta^{\nl}_{\mff_{c,d}(\lri)}(s)$ are known
for $c=1$ (see~\eqref{equ:c=1}), $c=2$ (see~\cite{Voll/05a}), and
$(c,d)=(3,2)$ (see~\cite[Theorem~2.35]{duSWoodward/08}).

\begin{remark}
  Theorem~\ref{thm:free} also implies a corrected version of the
  conclusion of~\cite[Theorem~4.73]{duSWoodward/08}. Note that the
  meaning of the numbers $N_1$ and $N_2$ there is different from the
  one in the present paper. In any case, the power of $p^{-s}$ on the
  right-hand side of \cite[(4.45)]{duSWoodward/08} is not equal to the
  sum of the coranks of the upper central series of $\mff_{c,d}$.
\end{remark}

\subsection{Some Lie rings of maximal class and their
  amalgams}\label{subsec:max.class}
Given \revone{an (integer)} partition $\lambda=(\lambda_1,\dots,\lambda_r)\in\N^r$, with
$\lambda_1 \geq \lambda_2 \geq \dots \geq \lambda_r$, define the
class-$\lambda_1$-nilpotent Lie ring
$$\mcL_{\lambda} = \la x_0, \{x_{i,j}\}_{i\in[r], j\in[\lambda_i]}
\mid \forall i\in[r], j\in[\lambda_j-1]:\; [x_0,x_{i,j}] =
x_{i,j+1}\ra_{\Z}.$$

For $r=1$ and $\lambda_1\geq 2$ this yields the Lie ring
$M_{\lambda_1}$ of maximal class $\lambda_1$ described
in~\cite[p.~99]{duSWoodward/08}. If $\lambda_r\geq 2$, then
$\mcL_{\lambda}$ is obtained from amalgamating $M_{\lambda_i}$,
$i\in[r]$, along~$x_0$. If $\lambda=1^{(r)}=(1,\dots,1)$, then we
obtain the abelian Lie ring $\Z^{1+r}$;
cf.\ Example~\ref{exa:c=1}. The general case is evidently just an
amalgamation of these special cases, again along~$x_0$. For
$\lambda=2^{(r)}$ we get ``Grenham's Lie rings'' $\mathcal{G}_{1+r}$
on $1+r$ Lie generators;
cf.\ \cite[Section~2.6]{duSWoodward/08}. Explicit formulae for the
local ideal zeta functions $\zeta^{\nl}_{\mcL_{2^{(r)}}(\Zp)}(s)$,
valid for all $p$, are given in \cite[Theorem~5]{Voll/05}. (On the
face of it, the formulae there are for the local normal zeta functions
of the torsion-free finitely generated nilpotent groups associated to
the Lie rings $\mcL_{2^{(r)}}$ by the Mal'cev correspondence. As
$c=2$, however, the formulae coincide with those of the ideal zeta
functions of $\mcL_{2^{(r)}}(\Zp)$ for all $p$; cf.\ \cite[Remark on
  p.~206]{GSS/88}.)

\begin{definition}
 We say that $\lambda$ is a \emph{near rectangle} if it is of the form
 $\lambda= (c^{(r_1)}, 1^{(r_2)})$ for some $c\in\N$ and
 $r_1,r_2\in\N_0$.  
\end{definition}

A geometric interpretation of this property of $\lambda$ is given in
Proposition~\ref{prop:interior}.  

\begin{lemma}\label{lem:near.rectangle}
 Condition~\ref{con} is satisfiable for $\mcL_{\lambda}$ if and only
 if $\lambda$ is a near rectangle.
\end{lemma}

\begin{proof}
 We may assume that $c>1$. If $\lambda = (c^{(r_1)},1^{(r_2)})$ is a
 near rectangle, then, setting
$$\mcL_i = \begin{cases} \la x_0,x_{1,1},\dots,x_{r_1,1}\ra_{\Z} &
   \textup{ for } i=1,\\ \la x_{1, i},\dots,x_{r_1,i}\ra_{\Z} &
   \textup{ for } i\in\, ]1,c-1],\\ \la x_{1,c},\dots,x_{r_1,c},x_{r_1+1, 1},\dots
  x_{r_1+r_2,  1}\ra_{\Z} & \textup{ for } i=c,\end{cases}$$ and
 $\{c_1,\dots,c_d\} = \{\ad(x_0),\ad(x_{i,1}) \mid i\in[r_1]\}$,
 Condition~\ref{con} is satisfied.

Assume now that $\lambda$ is not a near rectangle. This implies that
$c>2$ and there exists $k\in[r-1]$ which is minimal with respect to
the property that $c = \lambda_k > \lambda_{k+1}>1$. Write $\mcL =
\mcL_{\lambda}$, with upper central series $(Z_i)_{i=0}^c$. Following
\cite[Definition~4.49]{duSWoodward/08}, we define the \emph{depth} of
an element $x\in\mcL$ as
$$\dep(x) = c+1 - \min\{ i\in [c]_0 \mid x \in Z_i\}.$$ 

The depth-1-elements of the $\Z$-basis $\mcB = \{ x_0, x_{i,j} \mid
i\in [r], j\in [\lambda_i] \}$ of $\mcL$ are exactly $x_0$ and
$x_{\kappa, 1}$, for $\kappa\in [k]$. To see that Condition~\ref{con} is
not satisfiable, note that the subalgebra $\mcE_1 = \la \ad(z) \mid
z\in \mcB, \,\dep(z)=1\ra$ of the associative algebra $\mcE$ generated
by $\ad(\mcL)$ is a proper subalgebra of $\mcE$; indeed, the element
$x_{k+1,  1}$ has depth $c-\lambda_{k+1}+1>1$ and
$\ad(x_{k+1,  1})\in\mcE \setminus \mcE_1$. This follows, for instance,
from the fact that $\mcB$ clearly has the property that if $z,z'\in
\mcB$, then either $[z,z']=0$ or $\dep[z,z'] = \dep(z) + \dep(z')$.
\end{proof}

\subsubsection{Local ideal zeta functions}
The following consequence of Theorem~\ref{thm:main} implies, in
particular, the first of the two conclusions of
\cite[Proposition~4.75]{duSWoodward/08}. Note that the latter only
considers the case that $\lambda_r\geq 2$ and is conditional on the
incomplete~\cite[Conjecture~4.5]{duSWoodward/08}.

\begin{theorem}\label{cor:rectangle}
  Let $c\in\N$ and $r_1,r_2\in\N_0$.  For almost all primes $p$
  \revone{and all compact discrete valuation rings $\lri$ of residue
    characteristic~$p$,}
$$\left. \zeta^{\nl}_{\mcL_{(c^{(r_1)},
      1^{(r_2)})}(\lri)}(s)\right|_{p\rarr p^{-1}} =
  (-1)^{1+cr_1+r_2}q^{\binom{1+cr_1+r_2}{2} - s\left( c +
    \binom{c+1}{2}r_1 + r_2\right)} \zeta^{\nl}_{\mcL_{(c^{(r_1)},
      1^{(r_2)})}(\lri)}(s).$$ In particular,
  \cite[Conjecture~4.24]{duSWoodward/08} about the Lie rings
  $M_c$ of maximal class~$c$ holds.
\end{theorem}

\begin{proof} 
  By Lemma~\ref{lem:near.rectangle}, Condition~\ref{con} is
  satisfied. We find
  $\rk_{\Z}(\mcL_{(c^{(r_1)},1^{(r_2)})}) = N_0 = 1 + cr_1 + r_2$ and,
  more generally, $N_i = 1 + (c-i)r_1 + r_2 \delta_{i=0}$ for
  $i\in[c-1]_0$, whence
  $\sum_{i=0}^{c-1}N_i = c + \binom{c+1}{2}r_1 + r_2$. The result thus
  follows from Theorem~\ref{thm:main} \revone{and
    Corollary~\ref{cor:char.p}}.
\end{proof}

The second, conjectural conclusion of \cite[Proposition
4.75]{duSWoodward/08} suggests a positive answer to the following
question.

\begin{question} Is the ``near rectangle'' condition on $\lambda$
  necessary for local functional equations for the ideal zeta
  functions of the Lie rings~$\mcL_{\lambda}$?
\end{question}

A positive answer might also be hinted at by the explicit formulae for
$\lambda=(3,2)$ -- the smallest partition which is not a near
rectangle -- given in \cite[Theorem~2.32]{duSWoodward/08} as well as
the following results on the reduced zeta
functions~$\zeta^{\nl}_{\mcL_{\lambda},\red}(T)$; cf.\
Section~\ref{subsec:red}.

\subsubsection{Reduced ideal zeta functions}
\revtwo{Let $\mcL$ be a $\Ct$-Lie algebra with $\Ct$-basis
  $\mcB=\{b_1,\dots,b_d\}$. In \cite{Evseev/09}, the basis $\mcB$ is
  called \emph{simple} if, for all $1\leq i,j \leq d$, there exist
  $\ell\in[d]$ and $a\in\Ct$ such that $[b_i,b_j]=a b_{\ell}$. A pair
  $(b_i,b_j)$ of basis vectors is called \emph{removable} if there
  exist integers $\ell_1,\dots,\ell_d$ with $\ell_i\neq \ell_j$ such
  that, for all $z\in\C\setminus\{0\}$ and all $r\in[d]$, the maps
  $x_r \mapsto z^{\ell_r}x_r$ are automorphisms of $\mcL$. (Strictly
  speaking, \cite{Evseev/09} defines removability only for pairs of
  \emph{indices} of basis elements.) Finally, the basis $\mcB$ is
  called \emph{nice} if all pairs $(b_i,b_j)$ are removable;
  cf.\ \cite[Theorem~3.2]{Evseev/09} which establishes that this
  characterization is equivalent to the definition of niceness in
  \cite[Section~4]{Evseev/09}.}

\begin{proposition}\label{prop:nice&simple}
  The $\Ct$-Basis $\mcB_{\lambda} = \{ x_0, x_{i,j} \mid i\in [r],
  j\in [\lambda_i] \}$ of the $\Ct$-Lie algebra $\mcL_\lambda(\Ct)$ is
  nice and simple in the sense of \cite{Evseev/09}.
\end{proposition}

\begin{proof}
  $\mcB_{\lambda}$ is simple as, for all $x,y\in\mcB_\lambda$, there
  exist $z\in \mcB_\lambda$ and $\varepsilon\in\{0,1,-1\}$ such that
  $[x,y]= \varepsilon z$. To see that $\mcB_{\lambda}$ is nice we show
  that all pairs $(x,y)\in\mcB_\lambda^2$ with $x\neq y$ are
  removable.  Concretely, we need to find pairwise distinct integers
  $l_0,l_{i,j}$, $i\in[r]$, $j\in[\lambda_i]$, such that, for all
  $z\in\mathbb{C}\setminus \{0\}$, the maps
$$x_0 \mapsto z^{l_0}x_0, \quad x_{i,j} \mapsto z^{l_{i,j}}x_{i,j}, \quad
i\in[r], j\in[\lambda_i],$$ are automorphisms of
$\mcL_{\lambda}(\Ct)$. This is the case
whenever $l_0 + l_{i,j} = l_{i,j+1}$ for all $i\in[r]$,
$j\in[\lambda_i-1]$. Setting $l_0=1$ and $l_{i,j}=1 + \left(\sum_{\iota
    < i}\lambda_\iota\right) + j$ is one way to achieve this.
\end{proof}

Define the rational polyhedral cone
\begin{equation}\label{def:cone}
  \mcC_{\lambda} =  \left\{
    \bfn\in\R_{\geq0}^{1+\sum_{i=1}^r \lambda_i} \mid \forall
    i\in[r]:\;n_0,n_{i,1}\geq n_{i,2}\geq \dots \geq n_{i,\lambda_i}\right\}.
\end{equation}
\revtwo{Given $\bfn = (n_1,\dots,n_h)\in\R^h$ for some $h\in\N$, we set
  $\sum\bfn = \sum_{j=1}^h n_j$.}

\begin{corollary}
  $\zeta^{\nl}_{\mcL_\lambda,\red}(T) =
  \sum_{\bfn\in\mcC_{\lambda}\cap
    \N_0^{1+\sum_{i=1}^r\lambda_i}}T^{\sum \bfn}.$
\end{corollary}

\begin{proof} 
  As $\mcB_{\lambda}$ is nice and simple, this follows
  from~\cite[Proposition~4.1]{Evseev/09}, where $\mcC_{\lambda}$ is
  called~$\mcC_{\mcB_{\lambda}}^{\nl}$.
\end{proof}

The reduced zeta function $\zeta^{\nl}_{\mcL_\lambda,\red}(T)$ is thus
the Hilbert series associated to the graded monoid algebra
$R_{\mcC_{\lambda}}$ spanned by the integral points
in~$\mcC_{\lambda}$. Such Hilbert series satisfy functional equations
upon inversion of $T$ if and only if (!)  $R_{\mcC_{\lambda}}$ is
Gorenstein (cf.\ \cite[Theorem~12.7]{Stanley/96}). The latter
condition is satisfied if and only if there exists a unique minimal
integral vector $\bfbeta$ in the {interior} $\mcC_{\lambda}^{\circ}$
of $\mcC_{\lambda}$, i.e.\
\begin{equation}\label{equ:min.int}
\text{ if } \boldsymbol{\gamma}\in \mcC_{\lambda}^{\circ} \cap
\N^{1+\sum_{i=1}^r\lambda_i} \text{, then }\boldsymbol{\gamma} -
\bfbeta\in \mcC_{\lambda}.
\end{equation}
Note that $\mcC_{\lambda}^{\circ}$ is defined by replacing the
inequalities in \eqref{def:cone} by strict inequalities.
\begin{proposition}\label{prop:interior}
  A unique minimal integral interior vector
  $\bfbeta\in\mcC_{\lambda}^{\circ}$ exists if and only if $\lambda$
  is a near rectangle.
\end{proposition}

\begin{proof}
  Assume that $$\bfbeta = (\beta_0,\,
  \beta_{1,1},\beta_{1,2},\dots,\beta_{1\lambda_1},\,\beta_{21}, \dots
  \beta_{2\lambda_2},\,\dots,\,\beta_{r1},\dots,\beta_{r\lambda_r})
  \in \mcC_{\lambda}^{\circ} \cap \N^{1+\sum_{i=1}^r\lambda_i}$$ has
  the property \eqref{equ:min.int}. It is not hard to check that
\begin{equation}\label{equ:initial}
  (\beta_0,\, \beta_{1,1},\beta_{1,2},\dots,\beta_{1,\lambda_1}) =
  (\lambda_1,\, \lambda_1, \lambda_1-1,\dots,2,1).
\end{equation}
Observe that $\lambda$ is a near rectangle if and only if this
property determines $\bfbeta$ uniquely. Indeed, if $\lambda$ is a near
rectangle, say $\lambda=(c^{(r_1)},1^{(r_2)})$ for some $c\in\N$ and
$r_1,r_2\in\N_0$, then
\begin{multline*}
  \bfbeta = (c,\, c, c-1,\dots,2,1,\,c, c-1,\dots,2,1,\, \dots \, c,
  c-1,\dots,2,1,\, 1^{(r_2)})\\ \in
  \mcC_{(c^{(r_1)},1^{(r_2)})}^{\circ} \cap \N^{1+cr_1+r_2}
\end{multline*}
 satisfies both \eqref{equ:min.int} and \eqref{equ:initial} and is
 clearly unique with this property. If $\lambda$ is not a near
 rectangle, there exist more than one ways to ``complete'' the vector
 in \eqref{equ:initial} to an element $\beta \in
 \mcC_{\lambda}^{\circ} \cap \N^{1+\sum_{i=1}^r\lambda_i}$ such that
 $\beta_0=\lambda_1$ and $\beta_{i,1}=\lambda_1$ for $i\in[r]$. The
 difference between any two of these completions, however, is zero on
 these coordinates, hence is zero, contradiction.
\end{proof}

\begin{corollary}\label{cor:char.funeq.red}
  $\lambda$ is a near rectangle if and only if there exist $k,l\in\Z$
  such that
\begin{equation}\label{equ:funeq.red}
 \left. \zeta^{\nl}_{\mcL_\lambda,\red}(1/T)\right. = (-1)^lT^k
 \zeta^{\nl}_{\mcL_\lambda,\red}(T).
\end{equation}
In this case, $l \equiv 1 + c r_1+r_2\bmod(2)$ and $k =
c+\binom{c+1}{2}r_1+r_2$.
\end{corollary}

\begin{remark}
  If $\zeta^{\nl}_{\mcL_{\lambda}}(s)$ is almost uniform, say
  $\zeta^{\nl}_{\mcL_{\lambda}(\Zp)}(s)= W_{\lambda}(p,p^{-s})$ for
  $W_\lambda(X,Y)\in\Q(X,Y)$ and almost all primes $p$, then
  Corollary~\ref{cor:char.funeq.red} implies that the ``near rectangle
  condition'' on $\lambda$ is also necessary for local $\mfp$-adic
  functional equations. Indeed, in this case,
  $\zeta^{\nl}_{\mcL_{\lambda},\red}(T) = W_{\lambda}(1,T)$;
  cf.\ \cite[Section~3]{Evseev/09}.
\end{remark}

\begin{example}\label{exa:M4.vs.Fil4}
  A functional equation of the form \eqref{equ:funeq.red} for reduced
  zeta functions is not, in general, sufficient for functional
  equations for generic $\mfp$-adic zeta functions. Consider, for
  instance, the Lie rings $M_4=\mcL_{(4)}$ and $\Fil_4$;
  cf.\ Example~\ref{exa:Fil4}. Whereas the local ideal zeta functions
  $\zeta^{\nl}_{M_4(\Zp)}(s)$ satisfy functional equations
  (cf.\ Corollary~\ref{cor:rectangle}), the local ideal zeta functions
  $\zeta^{\nl}_{\Fil_4(\Zp)}(s)$ do not
  (cf.\ \cite[Theorem~2.39]{duSWoodward/08}). The respective reduced
  ideal zeta functions, however, coincide;
  cf.\ \cite[Example~4.5]{Evseev/09}.
\end{example}

\section{Further examples}\label{sec:further}
In this section we discuss three applications of
Theorem~\ref{thm:main} and Corollary~\ref{cor:char.p} to submodule
zeta functions which are not ideal zeta functions of nilpotent Lie
lattices. The case of one-generator matrix algebras is discussed in
Section~\ref{subsec:cyclic}. In Section~\ref{sec:abelian} we consider
a class of abelian matrix algebras. Section~\ref{sec:sut} is dedicated
to full algebras of strictly upper triangular matrices. Throughout,
$\lri$ is a compact discrete valuation ring \revone{of arbitrary
  characteristic} and residue field cardinality~$q$, \revone{a power
  of the residue characteristic $p$}, and let $n\in\N$.

\subsection{The case $d=1$}\label{subsec:cyclic}
In \cite{Rossmann/17}, zeta functions enumerating submodules invariant
under a single, not necessarily nilpotent endomorphism of a finitely
generated $\Gri$-module are studied. \cite[Theorem~A]{Rossmann/17}
gives, in particular, an explicit formula for almost all of the Euler
factors of such zeta functions in terms of translates of Dedekind zeta
functions of number fields. Both these number fields and the
combinatorics of the translations are determined by the rational
canonical form for the given
endomorphism. \cite[Theorem~B]{Rossmann/17} establishes generic
functional equations for the Euler factors. In the nilpotent case, to
which the general case is reduced in \cite[Section~3]{Rossmann/17},
they confirm Theorem~\ref{thm:main}.

\subsection{A class of abelian matrix algebras}\label{sec:abelian}
\subsubsection{Local formulae}\label{subsec:abelian.local}
Let $p$ be a prime and $\lri$ \revone{be a compact discrete valuation
  ring of residue characteristic $p$.} We set
\begin{equation}\label{equ:bfo}
  \mcM_{\bfo}(\lri) = \mcM_{(\underbrace{1,\dots,1}_{n \textup{ times}})}(\lri) = \left\{
    \left( \begin{matrix}0&\diag(z_1,\dots,z_n)\\0&0\end{matrix}\right)\mid
    z_1,\dots,z_n\in \lri\right\}\leq \Mat_{2n}(\lri)
\end{equation} 
and consider $\lri^{2n}$ as an $\mcM_{\bfo}(\lri)$-module by right
multiplication.  We record a formula for the submodule zeta functions
$\zeta_{\mcM_{\bfo}(\lri) \car \lri^{2n}}(s)$. It is very similar to
that for the ideal zeta functions $\zeta^{\nl}_{H(\lri)^{n}}(s)$ of
the $n$-fold direct product of the Heisenberg $\lri$-Lie lattices
$H(\lri)$ of strictly upper triangular $3\times3$-matrices over $\lri$
given in \cite[Theorem~3.1]{SV1/15}.  We formulate this result partly
in the notation of \cite{SV1/15}\revone{ and note that the relevant
  local results in \cite{SV1/15} are -- despite their formulation in
  characteristic zero only -- in fact valid for compact discrete
  valuation rings of arbitrary characteristic.}  In particular,
$\mcD_{2n}$ denotes the set of Dyck words in letters $\bfz$ and $\bfo$
of length $2n$, of cardinality $\frac{1}{n+1}\binom{2n}{n}$, the
$n$-th Catalan number. Recall the formula \eqref{equ:c=1} for
$\zeta_{\{0\} \car \lri^n}(s)$.

\begin{theorem}\label{thm:abelian.tot.split}
  \begin{equation*}
    \zeta_{\mcM_{\bfo}(\lri)\car \lri^{2n}}(s) = (1-q^{-s})^n
    \zeta_{\{0\} \car \lri^n}(s) \sum_{w\in \mcD_{2n}} D_w(q,q^{-s}),
  \end{equation*} 
  where, for $w = \prod_{i=1}^r ( \bfz^{L_i - L_{i-1}}
  \bfo^{M_i-M_{i-1}})\in\mcD_{2n}$, the rational function
  $D_w(q,q^{-s})$ is defined as $D_w^{\bfo}(p,p^{-s})$ in
  \cite[Theorem~3.1]{SV1/15} with $q$ in place of $p$ and with
  numerical data
\begin{alignat*}{2}
  x_j & = q^{j(n + L_i - j)-s(L_i + j)} &&\quad\text{ for } j
  \in \,]M_{i-1}, M_i] ,\\ y_j & = q^{(n -M_{i-1} + j)M_{i-1}-s(j +
    M_{i-1})} &&\quad\text{ for } j \in \,]L_{i-1}, L_i].
\end{alignat*}
\end{theorem}

\begin{proof}[Sketch of proof.] {Mutatis mutandis}, the analysis of
  \cite[Theorem~3.1]{SV1/15} carries over.  The centre
  $Z(H(\lri))\cong \lri^n$ is replaced by $Z_1(\lri) =
  \Cent_{\mcM_{\bfo}(\lri)}(\lri^{2n}) \cong \lri^n$, whereas the role
  of the cocentre $H(\lri) / Z(H(\lri))\cong \lri^{2n}$ is taken by
  $\lri^{2n} / Z_1(\lri)\cong \lri^n$, explaining the change in
  numerical data.
\end{proof}

The construction of the algebra $\mcM_{\bfo}(\lri)$ may be generalized
by replacing the diagonal matrices in \eqref{equ:bfo} by
``generic block-diagonal'' matrices with block sizes $f_1,\dots,f_g$,
say. Roughly speaking, formulae for the submodule zeta functions
associated to the resulting nilpotent algebras
$$\mcM_{\bff}(\lri)= \mcM_{(f_1,\dots,f_g)}(\lri) =
\left( \begin{matrix} 0 &
  \diag(\Mat_{f_1}(\lri),\dots,\Mat_{f_g}(\lri))\\0&0 \end{matrix}
\right) \leq \Mat_{2n}(\lri),$$ acting on $\lri^{2n}$ by right
multiplication, are obtained by modifying the ``numerical data'' in
the formulae given in \cite[Theorem~3.6]{SV1/15} and in the formula
for the zeta function preceding it. We leave the precise details to
the reader, spelling out only the result in the other ``extremal''
case $\bff = (n)$, yielding
$$\mcM_{(n)}(\lri) = \left( \begin{matrix} 0 & \Mat_{n}(\lri)
  \\0&0 \end{matrix} \right) \leq \Mat_{2n}(\lri).$$
The following is analogous to \cite[Corollary~3.7]{SV1/15}.

\begin{theorem}\label{thm:abelian.inert}
\begin{equation}\label{equ:abelian.inert}
  \zeta_{\mcM_{(n)}(\lri)\car \lri^{2n}}(s) = \zeta_{\{0\} \car
    \lri^n}(s)\frac{1}{1-x_n}\sum_{I\subseteq[n-1]}\binom{n}{I}_{q^{-1}}\prod_{i\in
    I}\frac{x_i}{1-x_i},
\end{equation} where $\binom{n}{I}_X\in\Z[X]$ denotes the
Gaussian multinomial coefficient and
$$x_j = q^{j(2n-j)-s(n+j)} \text{ for }j\in[n].$$
\end{theorem}

\begin{remark}
  Let $k$ be a field. By a theorem of Schur, $ \left( \begin{matrix}
    k\Id_n & \Mat_{n}(k)\\ 0 & k\Id_{n} \end{matrix}\right)$ is a
  maximal abelian subalgebra of $\Mat_{2n}(k)$;
  cf.\ \cite{Schur/05}. The right-hand side of
  \eqref{equ:abelian.inert} is a product of two \emph{Igusa functions}
  in the terminology of \cite[Definition~2.5]{SV1/15}.
\end{remark}

\begin{theorem}\label{thm:funeq.abel}
  For all $g\in\N$ and $\bff=(f_1,\dots,f_g)\in\N^g$, the functional
  equation
\begin{equation*}%\label{equ:funeq.abel}
  \left.\zeta_{\mcM_{\bff}(\lri) \car \lri^{2n}}(s)\right|_{\revtwo{p\rarr p^{-1}}}
  = q^{\binom{2n}{2}-3ns}\zeta_{\mcM_{\bff}(\lri)\car \lri^{2n}}(s)
\end{equation*}
holds.
\end{theorem}
\begin{proof}[Sketch of proof.] Analogous to \cite[Theorem~1.2]{SV1/15} with $3n$ and $5n(=3n+2n)$
replaced by $2n$ and $3n(=2n+n)$, respectively.
\end{proof}

\subsubsection{Global zeta functions and Euler products} The local
formulae in Section~\ref{subsec:abelian.local} may be put in a global
context as follows. For any ring $A$, consider
$$\mcM(A) = \left( \begin{matrix}0 & A \\ 0 & 0 \end{matrix}\right)
\leq \Mat_2(A).$$
Let $K$ be a number field with ring of integers $\Gri$ and of
degree~$n$, say. By restriction of scalars from $K$ to $\Q$, we may
consider $\mcM(\Gri)$ as a subalgebra of $\Mat_{2n}(\Z)$, turning
$\Z^{2n}$ into an $\mcM(\Gri)$-module by right multiplication, with
associated submodule zeta function
\begin{equation}\label{equ:euler.abelian}
  \zeta_{\mcM(\Gri) \car \Z^{2n}}(s) = \prod_{p \textup{ prime}}
  \zeta_{\mcM(\Gri\otimes_\Z\Zp) \car \Zp^{2n}}(s).
\end{equation}
This is a close analogue of the ideal zeta function
$\zeta^{\nl}_{H(\Gri)}(s)$ of the Heisenberg Lie ring over $\Gri$
studied in \cite{SV1/15}. If $p$ is unramified in~$K$, i.e.\ $p\Gri =
\prod_{i=1}^g \mfp_i$ for pairwise distinct prime ideals $\mfp_i$ of
$\Gri$ with residue degrees $f_i$ for $i\in[g]$, then
$\Gri\otimes_{\Z}\Zp \cong \prod_{i=1}^g \Z_{p^{f_i}}$, where
$\Z_{p^{f_i}}$ denotes the unramified extension of $\Zp$ of
degree~$f_i$. Hence
$$\mcM(\Gri\otimes_\Z\Zp) \cong \mcM_{\bff}(\Zp).$$ Therefore, all but
finitely many of the Euler factors in \eqref{equ:euler.abelian} are
covered by the formulae sketched -- and written out in
Theorem~\ref{thm:abelian.tot.split} for the primes which split totally
in $K$ and in Theorem~\ref{thm:abelian.inert} for the primes which
stay inert in $K$ -- in Section~\ref{subsec:abelian.local}. Formulae
for the Euler factors indexed by the rational primes which remain
unsplit in $K$ (but may ramify) may be obtained by modifying those for
$\zeta^{\nl}_{H(\Gri)}(s)$ described in \cite{SV2/14}. Note that
Theorem~\ref{thm:main} is applicable for $\mcE = \mcM(\Gri)$ and~$\mcL
= \Z^{2n}$ as $c=2$; cf.\ Remark~\ref{rem:con.sat}. The functional
equations established in Theorem~\ref{thm:funeq.abel} strengthen the
result in this special case by implying that the set of exceptional
primes is contained in (and conjecturally equal to) the set of primes
which ramify in~$K$.

\subsection{Strictly upper triangular matrices}\label{sec:sut} 
For $m\in\N$, let $\mfu_m(\lri)$ be the associative algebra of all
strictly upper triangular $m\times m$-matrices over $\lri$, acting on
$\lri^m$ by right-multiplication, say. Given a partition
$\lambda=(\lambda_1,\dots,\lambda_r)\in\N^r$ of
$n=\sum_{i=1}^r\lambda_i$, consider
$$\mfu_{\lambda}(\lri)=\bigoplus_{i=1}^r \mfu_{\lambda_i}(\lri),$$
diagonally embedded into $\mfu_{n}(\lri)$. Theorem~\ref{thm:main}
\revone{and Corollary~\ref{cor:char.p} are} clearly applicable and
impl\revone{y} the following result.

\begin{theorem}\label{thm:funeq.u}
  For almost all primes $p$ \revone{and all compact discrete valuation
    rings $\lri$ of residue characteristic~$p$,}
$$
\left.\zeta_{\mfu_{\lambda}(\lri)\car\lri^n}(s)\right|_{p\rarr p^{-1}}
= (-1)^n q^{\left(\binom{n}{2} - s\sum_{i=1}^r
    \binom{\revtwo{\lambda_i+1}}{2}\right)}
\zeta_{\mfu_{\lambda}(\lri)\car\lri^n}(s).$$
\end{theorem}

Explicit formulae for generic submodule zeta functions of the form
$\zeta_{\mfu_{\lambda}(\lri)\car\lri^n}(s)$ \revone{have been
  computed} for $\lambda=(m)$, $m\leq 5$ as well as for numerous other
partitions of natural numbers $n\leq 7$;
see~\cite[Section~9.4]{Rossmann/16} and the database in the computer
algebra package \cite{Rossmannzeta}. All these functions are given by
rational functions in $q$ and $q^{-s}$. Theorem~\ref{thm:funeq.u} is
consistent with and explains the functional equations recorded in
\cite[Theorems~9.5, 9.7, 9.8]{Rossmann/16}.

\section*{Acknowledgments} 
I am grateful for helpful conversations I had with Luke Woodward, over
and beyond the mathematical inspiration I gained from reading
\cite{duSWoodward/08}, a book whose influence on the paper I
acknowledge in detail in Section~\ref{subsec:nil.exa}. I thank Tobias
Rossmann for numerous discussions which had a profound impact on this
paper, rendering its results much more general and
conceptual. \revone{I am grateful to two anonymous referees whose
  comments helped me to improve the paper's exposition significantly
  and prompted Corollary~\ref{cor:char.p}. To Raf Cluckers I am
  particularly endebted for elucidations about the transfer principle
  which greatly helped me in the formulation of this corollary.} Work
on this paper was supported by DFG Sonderforschungsbereich~701 at
Bielefeld University.

%\bibliographystyle{amsplain}
%\bibliography{masterbibliography_february2014}

\def\cprime{$'$} \def\cprime{$'$}
\providecommand{\bysame}{\leavevmode\hbox to3em{\hrulefill}\thinspace}
\providecommand{\MR}{\relax\ifhmode\unskip\space\fi MR }
\providecommand{\MRhref}[2]{%
  \href{http://www.ams.org/mathscinet-getitem?mr=#1}{#2}
}
\providecommand{\href}[2]{#2}

\end{document}